\documentclass{article}
\usepackage{amssymb}
\usepackage{amsmath}
\usepackage{amsthm}
\usepackage{verbatim}
\usepackage{color}
\usepackage{url}
 \usepackage[all]{xy}
\usepackage{fancyhdr}
\usepackage{graphicx,amsmath,amsfonts,latexsym,amssymb,amsthm,color}
\usepackage{latexsym}
\usepackage{verbatim}
\usepackage{enumerate}
\usepackage{enumitem}

\usepackage[a4paper]{geometry}
\definecolor{blau}{rgb}{0.05,0.2,0.7}
\definecolor{auchblau}{rgb}{0.03,0.3,0.7}
\usepackage[pdftex,linkbordercolor={0.8 0.5 0.2}]{hyperref} 

\newtheorem{theorem}{Theorem}
\newtheorem*{theorem*}{Theorem}
\newtheorem{prop}{Proposition}
\newtheorem{lem}{Lemma}
\newtheorem{cor}{Corollary}
\newtheorem{defn}{Definition}

\newcommand{\R}{\mathbb{R}}
\newcommand{\C}{\mathbb{C}}
\newcommand{\N}{\mathbb{N}}

\renewcommand{\Im}{\operatorname{Im}}

\DeclareMathOperator{\supp}{supp}
\newcommand{\der}{\mathrm{d}}
\newcommand{\rmi}{\mathrm{i}}
\newcommand{\ee}{\mathrm{e} }

\newcommand{\rel}{\mathrm{rel}}

\newcommand{\free}{\textrm{free}}

\newcommand{\loc}{\mathrm{loc}}

\numberwithin{equation}{section}
 
\title{Dispersive Estimates for Maxwell's equations in the exterior of a sphere}
\author{Yan-long Fang \thanks{Department of Mathematics, University College London, WC1H 0AY, United Kingdom, yanlong.fang@ucl.ac.uk} \hspace{0.05cm}, \quad Alden Waters \thanks{Institute for Analysis, Department of Mathematics, Leibniz University Hannover, Welfengarten 1, 30167 Hannover, Germany, alden.waters@math.uni-hannover.de} \hspace{0.05cm}}

\begin{document}
\maketitle

\begin{abstract}
The goal of this article is to establish general principles for high frequency dispersive estimates for Maxwell's equation in the exterior of a perfectly conducting ball. We construct entirely new generalized eigenfunctions for the corresponding Maxwell propagator. We show that the propagator corresponding to the electric field has a global rate of decay in $L^1-L^{\infty}$ operator norm in terms of time $t$ and powers of $h$. In particular we show that some, but not all, polarizations of electromagnetic waves scatter at the same rate as the usual wave operator. The Dirichlet Laplacian wave operator $L^1-L^{\infty}$ norm estimate should not be expected to hold in general for Maxwell's equations in the exterior of a ball because of the Helmholtz decomposition theorem. 
\end{abstract}


\textbf{keywords:} Maxwell's equations, dispersive estimates \\

\textbf{MSC classes:} 35R01, 35R30, 35L20, 58J45, 35A22

\section{Introduction}
Determination of decay rates for Maxwell’s equations with time independent metrics is otherwise known as the Maxwell scattering problem. Maxwell's equations are vector-valued and in the presence of an obstacle mixed boundary conditions make it so they are not reducible to standard evolution equations even in the constant coefficient case. The Maxwell scattering problem is therefore more difficult than the wave or Helmholtz problem as the presence of a nontrivial topology can amplify or spread the initial pulse of waves.  This type of geometric analysis is important because Maxwell's equations are fundamental in electromagnetism, astrophysics, and nanotechnology. Here we are interested in the specific question of the rate of decay of these waves in time in the exterior of a ball. The property of a wave to spread out as time goes by while keeping its energy unchanged is called dispersion. In three dimensions it is claimed that light bends effectively at the same rate as the free space \cite{LI,LI2} (Euclidean, no obstacle case). Therefore there is no change in what is known as dispersive estimates in three dimensions for the ball. The wave equation more appropriately describes sound waves. But this important work \cite{LI} is the first to consider the case of dispersion with an obstacle present, in this case a sphere, and prove a non-local estimate. 

Light is not governed by the wave equation, but by Maxwell's equations. However surprisingly for electromagnetic radiation the global exterior dispersive decay rate remains unknown in the presence of obstacles. This decay rate is unknown even just for a sphere. In our case, we show using Mie scattering that a weaker than Euclidean $L^1\textendash L^{\infty}$ dispersive estimate still holds uniformly in the exterior of a sphere for the Maxwell propagator. This phenomenon has its roots in the Helmholtz decomposition theorem. A major goal of this article is to connect spectral theory for the Laplacian to Mie scattering theory for balls with perfectly conducting boundary conditions. In particular we construct generalized eigenfunctions in the exterior of a ball which solve the time-harmonic electric field problem. The idea is that the spectral theory solutions are fixed by the boundary conditions on the object and their behavior at infinity. These representations are exact for the ball. We then use these spectral theory representations to describe the kernels of evolution operators. 

The study of dispersive estimates has a long history in the function case. Exterior dispersion was studied for the Schr\"odinger and wave case in \cite{LI, LI2, VK}. While energy estimates have been done for Maxwell in a variety of mixed spaced settings, genuine $L^{1}-L^{\infty}$ estimates have not been examined. In the non-compact case in general geometric settings, $L^r-L^q$ for $r>1$ estimates for $p$-forms has been examined in the work of \cite{gs} and Maxwell's equations in \cite{AS}. Resolvent estimates which are related for $p$-forms can be found in \cite{gh1,gh2} and the often used here \cite{OS}. The technique of constructing a generalized eigenfunction is similar to those asymptotic expansions used to establish general scattering rates for functions in \cite{gh3,Parnovski}. Related literature in the case of convex bodies for the Dirichlet Laplacian includes \cite{Zscatter, HL, IS, dl} and Mie scattering in dimension 2, \cite{Galk}. Strichartz estimates in some related cases for Maxwell's equations are established in \cite{Sc1,Sc2}. We mention also that recently Strichartz estimates have been established for Lipschitz domains for the interior Maxwell problem in \cite{bs} using the language of $p$-forms in the case of non-constant coefficients. One can see also \cite{sp1, sp2, sp3} for recent progress on Strichartz estimates in the interior case with variable coefficients. The estimates here could be useful for establishing existence and uniqueness for semi-linear Maxwell's equations as done in the Dirichlet function case of the exterior of convex bodies  \cite{hs}. We hope that the construction of the kernel using spectral theory could be useful in the related examination of Price's law \cite{hintz,wunsch} in the future.  
 
\section{Precise setup and notation} \label{sec:1.1}

We consider the space $\mathbb{R}^3$ and a single smooth obstacle $\mathcal{O}$, an open ball, removed from the space. We set $\mathbb{R}^3\setminus \mathcal{O}=M$. The time harmonic Maxwell system is given by 
\begin{align}\label{system}
&\mathrm{curl} E-\rmi \lambda H =0 \\ \nonumber
&\mathrm{div}E =0 \\  \nonumber
& \mathrm{curl}H+\rmi \lambda E=0 \\ \nonumber
& \mathrm{div}E=0\\ \nonumber
& \nu\times E=A \\ \nonumber
& \nu\cdot H=f
\end{align}
where the first four equations hold in $M$. Here $\nu$ is the everywhere defined outward pointing unit normal vector field on $\partial\mathcal{O}$. We assume that $\lambda\in \mathbb{R}$ unless otherwise stated. The system for $E$ then becomes
\begin{align}\label{esystem}
&\Delta E-\lambda^2 E=0 \quad \mathrm{in}\quad M \\
&\mathrm{div} E=0 \quad \mathrm{in}\quad M \nonumber\\
& \nu\times E=A \quad \mathrm{on}\quad \partial\mathcal{O}. \nonumber
\end{align}
The associated spectral theory problem is that of the Laplace-Beltrami operator $\Delta$ (with positive spectrum) on divergence-free vector fields with the corresponding boundary condition, c.f. intro to \cite{AWE} for a complete exposition. We use sign convention for the Laplacian which gives it positive spectrum and are careful to translate any results from the analysis literature to this convention.  The boundary conditions $\nu\times E=0$ for the electric field on $\partial\mathcal{O}$ leads to the relative Laplacian $\Delta_{\rel}$ by quadratic form considerations. Similarly for the magnetic field the boundary condition $\nu\cdot H=0$ leads to the absolute Laplace operator $\Delta_{\mathrm{abs}}$. Both are self-adjoint operators on $L^2(M,\mathbb{C}^3)$. 

We now recall some definitions. The set of distributions $\mathcal{D}'(M,E)$ on a Riemannian manifold taking values in a real or complex vector space E is the vector space of continuous linear functionals on $C_0^{\infty}(M,E^*)$. We understand $C^{\infty}(M,E)$ as the subspace of $\mathcal{D}'(M,E)$ using the bilinear pairing $C^{\infty}(M,E)\times C_0^{\infty}(M,E^*)\mapsto \mathbb{C}$ induced by integration with respect to volume. For an open set $\mathcal{U}\subset\mathbb{R}^3$, the standard Sobolev spaces $L^2$ are denoted as $H^s(\mathcal{U})$ or $H^s(\mathcal{U})$. We write $H^s_{comp}(\mathcal{U})$ for the space of elements in $H^s(\mathcal{U})$ with compact support, to distinguish them from the closure of $H_0^s(\mathcal{U})$ in $H^s(\mathcal{U})$. The local Sobolev spaces $H^s_{loc}(\mathcal{U})$ are defined, as usual, as the space of distributions $\phi\in \mathcal{D}'(U)$ such that $\chi\phi\in H^s(U)$ for all $\chi\in C_0^{\infty}(\mathcal{U})$.  

Let $D\subset M$ be a bounded open set. We now define the following spaces of vector fields that are standard in Maxwell theory
\begin{align*}
&H(\mathrm{curl},D)=\{f\in L^2(D,\mathbb{C}^3)|\,\,\mathrm{curl} f\in L^2(D,\mathbb{C}^3)\}\\
&H^{-\frac{1}{2}}(\mathrm{Div}\, 0,\partial\mathcal{O})=\{f\in H^{-\frac{1}{2}}(\partial\mathcal{O};T\partial\mathcal{O})| \,\,\mathrm{Div} f=0\} \\
&H(\mathrm{curl}\,0, D)=\{ u\in H(\mathrm{curl},D): \mathrm{curl}\,u=0\,\, \textrm{in}\,\, D\} \\
&H_0(\mathrm{curl}\,0,D)=\{ u\in H_0(\mathrm{curl},D): \mathrm{curl}\,u=0\,\, \textrm{in}\,\, D\}\\
&\nabla H_0^1(D)=\{\nabla p: p\in H_0^1(D)\}\\
&H_0(\mathrm{curl}, \mathrm{div}\, 0, D)=\{v\in H_0(\mathrm{curl},D):  \langle v,\nabla\varphi\rangle_{L^2(D)}=0 \,\,\forall \varphi\in H_0^1(D)\}\\
&L^2(\mathrm{div}\, 0,D)=\{v\in L^2(D,\mathbb{C}^3): \langle v,\nabla\varphi\rangle_{L^2(D)}=0\,\, \forall \varphi\in H_0^1(D)\}.
\end{align*}
We also remind the reader that for any bounded open set $D\subset M$ the space $H_0(\mathrm{curl},D)$ is the closure of $C_0^{\infty}(D,\mathbb{C}^3)$ vector fields in $H(\mathrm{curl},D)$.

The corresponding resolvent for the Laplace-Beltrami operator we denote as $R_{\lambda}$, formally $R_{\lambda}=(\Delta_{\rel}-\lambda^2)^{-1}$. By domain considerations, $R_{\lambda}: L^2_{comp}(M,M^3)\rightarrow H^2_{loc}(M,M^3)$ c.f. Appendix B \cite{OS}. We also recall as part of the Helmholtz decomposition theorem (as stated by Theorem 4.23 in \cite{KH}): 
\begin{align}\label{Helmholtzdecomp}
H_0(\mathrm{curl},D)
= H_0(\mathrm{curl}, \mathrm{div}\, 0, D)\oplus \nabla H_0^1(D)
\end{align}
and
\begin{align}\label{L2decomp}
L^2(D,\mathbb{C}^3)=L^2(\mathrm{div}\,0,D)\oplus \nabla H_0^1(D).
\end{align}

\section{Statement of the main theorem}
We start with the following definition.
\begin{defn}
We say a cut-off function is:
\begin{itemize}
\item A high frequency cut-off function if $\varphi \in C^\infty_0(\R^+)$ is a non-negative function with support in a neighbourhood of $\lambda=a$ for some positive $a$ with $\varphi(\lambda)=1$ in a smaller neighbourhood of $\lambda=a$. Without loss of generality, we can assume $\supp{\varphi}\subset (\frac{1}{2}a,\frac{3}{2}a)$.
\item A low frequency cut-off function if $\varphi \in C^\infty(\overline{\R^+})$ is a non-negative function and satisfies $\varphi(0)=1$ while still having compact support.
\end{itemize}
\end{defn}
In our estimates we only consider high frequency cutoffs. This avoids some additional low frequency Bessel function expansions which are technical and we will report on these at a later time. When there is no obstacle in the problem, the spectral representation of the wave operator for the Laplacian on vector fields can be calculated explicitly in terms of Bessel functions. For the free case, we use the subscript $\free$ to emphasise this fact for various quantities. Our main theorem, Theorem \ref{spheretheorem2}, is for $M=\mathbb{R}^3\setminus\mathcal{O}$ where $\mathcal{O}$ is an open ball of fixed radius $\rho\geq 1$ with perfectly conducting boundary conditions. Let $R$ be a constant and $M\cap B_{R}(0)=M_{R}$ where $B_{R}(0)$ is a ball of radius $R$. We consider the domain $H_0(\mathrm{curl}, \mathrm{div}\, 0, M_{3R})$, as the underlying domain for the operators in question so that the estimates make sense. The constant $C$ in all of the theorems depends on at most the dimension, $a$ and $\rho$. The interpretation of our Maxwell propagator in 3 dimensions with relative boundary conditions is the electric field with constant dielectric and permittivity constants. Let $D_t=-i\partial_t$, our main theorem is then:
 
\begin{theorem}\label{spheretheorem2}  Let $d=3$ and $\mathcal{O}$ be a ball of radius $\rho$ centered at the origin. Let $R$ be a constant and $M\cap B_{R}(0)=M_{R}$ where $B_{R}(0)$ is a ball of radius $R$. Furthermore let $\varphi$ be a high frequency cutoff, $\rho\geq \frac{2}{a}$, $a>h, R>\rho\geq 1$. Let $C$ be a nonzero constant depending on $\rho$ and $a$ only. The following dispersive estimate holds with $\Delta_{\mathrm{rel}}$ for any $h\in (0,\frac{1}{4})$:
\begin{align}
&\left\|\varphi(hD_t)\cos(t\sqrt{\Delta_{\mathrm{rel}}})\right\|_{L^1(M_{R},M_{R}^3)\rightarrow L^{\infty}(M_{3R},M_{3R}^3)}
\leq Ch^{-5}\mathrm{min}\left\{1,\left(\frac{R}{|t|}\right)\right\}
\end{align}
whenever the underlying domain for the propagator is also in $H_0(\mathrm{curl}, \mathrm{div}\, 0, M_{3R})$.
\end{theorem}
If the domain of the propagator is the space of $f$ with $f\in H^{1}_{comp}(M;\mathbb{R}^3)$ which are divergence free on $M$ in the weak sense, we also have the following corollary
\begin{cor}\label{indepcor}
We have that 
\begin{align}
\left\|\varphi(hD_t)e^{\rmi t\sqrt{\Delta_{\mathrm{rel}}}}\right\|_{L^1(M,\mathbb{C}^3)\rightarrow L^{\infty}(M,\mathbb{C}^3)}
\leq Ch^{-5}\mathrm{min}\left\{1,\left(\frac{1}{|t|}\right)\right\}
\end{align}
where $C$ depends only on $\rho$ and $a$.
\end{cor}
The goal of this paper is to analyze the mechanism behind the Maxwell scattering versus wave scattering for the exterior of a single ball. The basic theory of Mie scattering to cope with relative boundary conditions is in Section \ref{ballscatter}. Here we also develop generalized eigenfunctions for Maxwell's equations which have a specific form and are exact representations of the solution $E$ to \eqref{esystem} for the exterior of a ball. The later section, Section \ref{laterproofs} proves Theorem \ref{spheretheorem2}, and Section \ref{finitespeed} improves the result to a global one by presenting the proof of Corollary \ref{indepcor}, which is built on the works of \cite{melroseT, melroseT2, melroseT3} and \cite{LP}. The results have not yet been shown to be sharp in terms of $R$ and powers of $h$, but the major proof contribution is that the associated propagator can be compared with the free space propagator to give a generic bound in terms of powers of $h$ and $R$. This comparison could be used in future work in the acoustic scattering case as well where in higher dimensions it is not yet known what the dispersive estimate is for a ball in dimensions higher than 3. 

\section{Scattering estimates in the case of the ball}\label{ballscatter}

\subsection{Generalized Eigenfunctions for Maxwell's equations}
For Maxwell's equations we are only using the case $d=3$. For these equations there is a special orthonormal basis of spherical harmonics. These allow us to calculate scattering matrices more easily than those from \cite{OS} as in this basis the scattering matrix is diagonal in the exterior of a sphere. Let $\ell\in \mathbb{N}$ and $|m|\leq \ell$. We start by defining the ordinary scalar harmonics
\begin{align}
Y_{\ell m}(\theta,\phi)=\sqrt{\frac{\epsilon_{m}}{2\pi}}\sqrt{\frac{2\ell+1}{2}\frac{(\ell-|m|)!}{(\ell+|m|)!}}P_{\ell}^{|m|}(\cos\theta)\begin{cases}\cos m \phi \\ \sin m\phi \end{cases}
\end{align}
where the Neumann factor $\epsilon_m$ is defined as
\begin{align}
\epsilon_0=1 \quad \epsilon_m=2 \quad m>0.
\end{align} 
The Legendre functions are reviewed in the Appendix, equation \eqref{legendre}, and the text below it. 

For our calculations and representation theorems to hold with vector spherical harmonics we would have to construct new generalized eigenfunctions for Maxwell's equations which are different from the generalized eigenfunctions constructed for the $p$-form Laplacian in \cite{OS}. However we have to make some significant changes to definitions of the generalized eigenfunctions because the vector spherical harmonics used in \cite{OS, RBM} do not solve the same equations that the harmonic polynomials on the sphere do. 

We start with the definition of the vector spherical harmonics, sometimes abbreviated VSH in the literature. Let $x\in \mathbb{R}^3$ so that $x=r\hat{r}$ where $\hat{r}=(\sin\theta \cos\phi, \sin\theta \sin\phi, \cos\theta)$ and $r>0$ is a scalar. Now we can define the vector spherical harmonics on the surface $\{x\in\mathbb{R}^3: |x|=r\}$ as follows
\begin{align}
\Psi_{1,\ell,m}=\frac{1}{\sqrt{\ell(\ell+1)}}\nabla Y_{\ell m}\times x \qquad \Psi_{2,\ell,m}=\frac{r}{\sqrt{\ell(\ell+1)}}\nabla Y_{\ell m} \qquad \Psi_{3,\ell,m}=\hat{r}Y_{\ell m}.
\end{align}
We re-label them simply as $Y_{\mu}$ when we restrict to the sphere $r=1$- this is elaborated on in the Appendix subsection \ref{VSH}.  The index $\mu$ is then $(j,\ell,m)$ where $j=1,2,3$, $\ell\in \mathbb{N}$ and $|m|\leq \ell$. These VSH form an orthonormal basis for the space $L^2(\mathbb{S}^{2},\mathbb{C}^3)$. Let $\mathcal{H}^1_{\rel}(\mathbb{S}^2)$ denote the space of $L^2(\mathbb{S}^2,\mathbb{C}^3)$ harmonic vector fields. We note that the the $\Psi_{1,\ell,m}$ are divergence free on the sphere, the $\Psi_{2,\ell,m}$ are curl free on the sphere and $\Psi_{3,\ell,m}\in \mathcal{H}^1_{rel}(\mathbb{S}^2)$, corresponding to the Helmholtz decomposition theorem (We have restated this as equation \eqref{L2decomp} from Theorem 4.23 in \cite{KH}). Notice that $\Psi_{1,\ell,m}$ is harmonic of degree $\ell$, whereas $\Psi_{2,\ell,m}$ and $\Psi_{3,\ell,m}$ are linear combinations of spherical harmonics of degree $\ell-1$ and $\ell+1$. This will become important for later computations. We supress the subscript $d$ where it is understood as when we refer to the VSH $d=3$. We define the space $L^2_t(\mathbb{S}^2,\mathbb{C}^3)$ as 
\begin{align}
L^2_t(\mathbb{S}^2,\mathbb{C}^3)=\{ \nu \cdot f=0\,\,\,\mathrm{on}\,\,\, \mathbb{S}^2\}. 
\end{align} 
Then we further have $\Psi_{j,\ell,m}$ with $j=1,2$ form an orthonormal basis for the subspace $L^2_t(\mathbb{S}^2,\mathbb{C}^3)$, due to the Helmholtz decomposition theorem. These properties of the $\Psi_{j,\ell,m}$ are reviewed in Theorem 2.46 and Corollary 2.47 of \cite{KH}.

A main observation of this paper is to establish the relation between Mie scattering and spectral representations for operators governing the evolution of Maxwell equations. When we consider perfectly conducting boundary conditions, we will be able to reconstruct a version of Stone's theorem in this coordinate system. Mie scattering has a long history of successful use in modelling electromagnetic fields outside of obstacles starting with the work of Gustav Mie \cite{mie}. 
We start with the free space problem relationship to the Helmholtz equation. We see that
\begin{lem}[Lemma 2.42 in \cite{KH}]\label{KHlemma}
Let $\lambda\ne 0$. If $u(x)$ is a scalar solution to the Helmholtz equation in $\mathbb{R}^3$, that is  $\Delta u-\lambda^2u=0$, then the pairs $(E_1,H_1)$ and $(E_2,H_2)$ where 
\begin{align*}
E_1(x)=\mathrm{curl}(xu(x))\quad  \mathrm{and}\quad
H_1(x)=\frac{1}{\rmi\lambda}\mathrm{curl}E_1(x)
\end{align*}
and 
\begin{align*}
E_2(x)=\mathrm{curl}\,\mathrm{curl}(xu(x)) \quad \mathrm{and} \quad
H_2(x)=\frac{1}{\rmi\lambda}\mathrm{curl}E_2(x)
\end{align*}
are solutions to the time-harmonic Maxwell system \eqref{system} whenever $\mathcal{O}=\emptyset$, on open subsets $D$ of $\mathbb{R}^3$.
\end{lem}

Indeed, when there is no obstacle, Bessel functions can be used in this way to construct a solution to the time-harmonic Maxwell system \eqref{system}. The $'$ notation refers to the $z$ derivative. This is used in the Maxwell solution formulas for space reasons. This is taken from \cite{KH}, and a similar formula is found in \cite{scat}.  
\begin{theorem}[Theorems 2.43 in \cite{KH}] \label{K2.43}
Let $x=r\hat{r}$. The vector field $E_1: \mathbb{R}^3\rightarrow \mathbb{R}^3$ defined by 
\begin{align*}
E_1(x)=\sqrt{\ell(\ell+1)}\mathrm{curl}(xj_{\ell}(\lambda r)Y_{\ell m})=j_{\ell}(\lambda r)\Psi_{1,\ell,m}
\end{align*}
and the corresponding vector field 
\begin{align*}
H_1(x)=\frac{\mathrm{curl}E(x)}{\rmi\lambda}=\frac{\ell(\ell+1)}{\rmi\lambda r}j_{\ell}(\lambda r)\Psi_{3,\ell,m}+\frac{\sqrt{\ell(\ell+1)}}{\rmi\lambda r}(zj_{\ell}(z))'|_{z=\lambda r}\Psi_{2,\ell,m}
\end{align*}
solve the corresponding Maxwell system \eqref{system}.\\ Analogously the vector fields 
\begin{align*}
&E_2(x)=\mathrm{curl}\,\mathrm{curl}(xj_{\ell}(\lambda r)Y_{\ell m})=\\& \ell(\ell+1)\frac{j_{\ell}(\lambda r)}{r}\Psi_{3,\ell,m}+\sqrt{\ell(\ell+1)}\frac{(zj_{\ell}(z))|_{z=\lambda r}'}{r}\Psi_{2,\ell,m}
\end{align*}
 and $$H_2(x)=\frac{\mathrm{curl}E(x)}{\rmi\lambda}$$ also solve the Maxwell system \eqref{system}. 
\end{theorem} 
\begin{proof}
We sketch the proof of the first part directly from \cite{KH} only for completeness. Using the fact that $\mathrm{curl}x=0$ and $$\mathrm{curl}(\lambda F)=(\nabla \lambda)\times F+\lambda\mathrm{curl}F$$ for $F$ and $\lambda$ scalar functions, we have that 
\begin{align*}
\mathrm{curl}(xj_{\ell}(\lambda r)Y_{\ell m})&= \nabla(j_{\ell}(\lambda r)Y_{\ell m})\times (r\hat{r})\\
& =\left(\frac{1}{r}j_{\ell}(\lambda r)\mathrm{Grad}_{\mathbb{S}^2}Y_{\ell m}+\frac{\partial}{\partial r}\left(j_{\ell}(\lambda r)\right) Y_{\ell m} \hat{r}\right)\times x. 
\end{align*}
This gives the first part. For the second part we have that $\mathrm{curl}\mathrm{curl}=\Delta+\nabla \mathrm{div}$ and the fact that $w=j_{\ell}(\lambda r)Y_{\ell m}$ solves the Helmholtz equation. We see that 
\begin{equation*}
\mathrm{curl}\,\mathrm{curl}(xw(x))= 
\Delta(xw(x))+\nabla \mathrm{div}(xw(x))=-2\nabla w(x)+\lambda^2 xw(x)+3\nabla w(x)+\nabla(x\cdot\nabla w(x))
\end{equation*}
which when all the derivatives are written out can be simplified using Bessel's equation \eqref{eqdef} and gives the desired result. Complete details are in \cite{KH}. 
\end{proof}  
We would like to show the following analogue of Proposition 1.2 in \cite{RBM} is applicable in three dimensions where $\Phi$ there which represents a vector of harmonic polynomials on the sphere is replaced by $Y$, a vector written in the basis of VSH. 

Suppose that $f,g: Z\rightarrow W$ are functions that take values in a locally convex topological vector spaces and $h:Z\rightarrow \mathbb{R}$. As usual we write $f=g+O_W(h)$ if for every continuous semi-norm p on $W$ there exists a constant $C_p$ such that $p(f(z)-g(z))\leq C_p|h(z)|$ for all $z\in Z$. 
 \begin{prop}[Analogue Proposition 1.2 in \cite{RBM}]\label{1.2}
For fixed $\lambda>0$ the generalised eigenfunctions $E_{\lambda}(Y)$ are completely determined by the following
\begin{enumerate}
\item $(\Delta_{\rel}- \lambda^2 ) E_{\lambda}(Y) =0$,
\item $\chi E_{\lambda}\in \mathrm{dom}(\Delta_{\rel})$ for any $\chi \in C^\infty_0(M)$ with $\nabla\chi=0$ near $\partial \mathcal{O}$,
\item The asymptotic expansion
$$
 E_{\lambda}(Y) = \frac{\ee^{-\rmi \lambda r} \ee^{\frac{i\pi}{2}} }{r} Y + \frac{\ee^{\rmi \lambda r} \ee^{-\frac{i\pi}{2}} }{r} \Psi_\lambda + O\left(\frac{1}{r^{2}}\right),  \quad \textrm{for} \,\,\,r \to \infty.
$$
for some $\Psi_\lambda \in C^\infty(\mathbb{S}^2;\C^3)$.
\end{enumerate}
\end{prop}
As a result in three dimensions $\Psi_\lambda$ is uniquely determined and implicitly defines a linear mapping 
\begin{gather*}
 S_\lambda:  C^\infty(\mathbb{S}^2; \C^3) \to C^\infty(\mathbb{S}^2; \C^3), \quad Y \mapsto \tau \Psi_\lambda,
\end{gather*}
where $$\tau: C^\infty(\mathbb{S}^2;\C^3) \to C^\infty(\mathbb{S}^2;\C^3)$$ is the pull-back of the antipodal map. The map $S_\lambda : C^\infty(\mathbb{S}^2, \C^3) \to C^\infty(\mathbb{S}^2, \C^3)$ is called the scattering matrix,
and $ A_\lambda= S_\lambda - \mathrm{id}$ is called the scattering amplitude. Similar to the Helmholtz decomposition, the scattering matrix is then of the form
\begin{align}\label{decomp}
 S_\lambda = \left( \begin{matrix} S_{n,\lambda} & 0 \\ 0 & S_{t,\lambda} \end{matrix} \right),
\end{align}
if $C^\infty(\mathbb{S}^2,\C^3)$ is decomposed as $C^\infty_n(\mathbb{S}^2,\C^3) \oplus C^\infty_t(\mathbb{S}^2,\C^3)$ into normal and tangential parts. 
If we had this proposition, then this will allow us to define  divergence free generalized eigenfunctions for the electric field in $M$. In order to try to prove this analogue of Proposition 1.2 from \cite{RBM} we use Theorem 2.43 (stated here as Theorem \ref{K2.43}) in \cite{KH}.  We make the definitions
\begin{defn}
For $Y\in L^2_t(\mathbb{S}^2,\C^3)$, let $a_{j,\ell,m}= \langle Y,\Psi_{j,\ell,m}\rangle$, $x=r\hat{r}$ as in Theorem \ref{K2.43} (Theorem 2.43 from \cite{KH})  then we define 
\begin{align}
&\tilde{j}_{\lambda,Y}(Y)= 2\sum \limits_{\ell,m} a_{1,\ell,m}\Psi_{1,\ell,m}\lambda j_{\ell}(\lambda r)(-\rmi)^{\ell} +\\& \nonumber a_{2,\ell,m}\Psi_{2,\ell,m}\frac{\left(z j_{\ell}(z)\right)'}{r}|_{z=\lambda r}(-\rmi)^{\ell-1}+
a_{2,\ell,m}\Psi_{3,\ell,m}\frac{\sqrt{\ell(\ell+1)}j_{\ell}(\lambda r)}{r}(-\rmi)^{\ell-1} \\&=
2\sum\limits_{\ell,m} \left(a_{1,\ell,m}(-\rmi)^{\ell}\mathrm{curl}(xj_{\ell}(\lambda r)Y_{\ell m})+a_{2,\ell,m}(-\rmi)^{\ell-1}\mathrm{curl}\mathrm{curl}(xj_{\ell}(\lambda r)Y_{\ell m})\right)
\end{align}
and also 
\begin{align}
&\tilde{h}^{(2)}_{\lambda,Y}(Y)=\sum _{\ell,m}a_{1,\ell,m}\Psi_{1,\ell,m}\lambda h^{(2)}_{\ell}(\lambda r)(-\rmi)^{\ell} +\\& \nonumber a_{2,\ell,m}\Psi_{2,\ell,m}\frac{\left(zh^{(2)}_{\ell}(z)\right)'}{r}|_{z=\lambda r}(-\rmi)^{\ell-1}+
a_{2,\ell,m}\Psi_{3,\ell,m}\frac{\sqrt{\ell(\ell+1)}h^{(2)}_{\ell}(\lambda r)}{r}(-\rmi)^{\ell-1}=\\&
\sum\limits_{\ell,m} \left(a_{1,\ell,m}(-\rmi)^{\ell}\mathrm{curl}(xh^{(2)}_{\ell}(\lambda r)Y_{\ell m})+a_{2,\ell,m}(-\rmi)^{\ell-1}\mathrm{curl}\mathrm{curl}(xh^{(2)}_{\ell}(\lambda r)Y_{\ell m})\right)
\end{align}
and
\begin{align}
&\tilde{h}^{(1)}_{\lambda,Y}(Y)=\sum _{\ell,m}a_{1,\ell,m}\Psi_{1,\ell,m}\lambda h^{(1)}_{\ell}(\lambda r)(-\rmi)^{\ell} +\\& \nonumber a_{2,\ell,m}\Psi_{2,\ell,m}\frac{\left(zh^{(1)}_{\ell}(z)\right)'}{r}|_{z=\lambda r}(-\rmi)^{\ell-1}+
a_{2,\ell,m}\Psi_{3,\ell,m}\frac{\sqrt{\ell(\ell+1)}h^{(1)}_{\ell}(\lambda r)}{r}(-\rmi)^{\ell-1}=\\&
\sum\limits_{\ell,m}\left(a_{1,\ell,m}(-\rmi)^{\ell}\mathrm{curl}(xh^{(1)}_{\ell}(\lambda r)Y_{\ell m})+a_{2,\ell,m}(-\rmi)^{\ell-1}\mathrm{curl}\mathrm{curl}(xh^{(1)}_{\ell}(\lambda r)Y_{\ell m})\right)
\end{align}
where these sums are defined when they converge in $C^{\infty}(M,\mathbb{C}^3)$.
\end{defn} 
We can then define generalized eigenfunctions using the basis of vector spherical harmonics, $E_{\lambda}(Y)$, in the exterior of $\mathcal{O}$ an open ball of radius $\rho$ which without loss of generality we assume $\rho\geq 1$. We pick a smooth function $\chi\in C^{\infty}(M)$ such that $1-\chi$ is compactly supported. We let $U$ denote the set $B_R(0)\setminus B_{\rho}(0)$ where $B_R(0)$ is a ball of fixed radius $R,$ with $R>\rho$. Then, we start with the following proposition.
\begin{prop}\label{cvim}
 We define an eigenfunction to $\Delta_{\rel}$ on $M$ by setting
\begin{align}
\chi \tilde{j}_{\lambda,Y}(Y)-R_{\lambda}(\Delta_{\rel}-\lambda^2)(\chi \tilde{j}_{\lambda,Y}(Y))=E_{\lambda}(Y)
\end{align}
for any $\chi\in C_0^{\infty}(M)$ and this is divergence free. Moreover for $\lambda\in\mathbb{C}$ with $|\lambda|<C$ where $C$ is a constant we have that on $U$
\begin{align}
E_{\lambda}(Y_{n})=O_{C^{\infty}(\mathbb{C}^3\setminus 0)}\left(\frac{\lambda^{\ell_{n}-1}R^{\ell_{n}}}{\Gamma\left(\ell_{n}+\frac{3}{2}\right)}\right)
\end{align}
so that $E_{\lambda}(Y)\in C^{\infty}_0(M,\mathbb{C}^3)$. 
\end{prop} 
\begin{proof}
We have that $E_{\lambda}(Y)$ solves the desired equation by construction. This proof follows similarly to Lemma 2.9 in \cite{OS} and is included for completeness. We have for $\lambda$ with $|\lambda|<C$ that 
\begin{align}
\tilde{j}_{\lambda,Y}(Y_{n})=O_{C^{\infty}(\mathbb{C}^3\setminus 0)}\left(\frac{\lambda^{\ell_{n}+1}R^{\ell_{n}}}{\Gamma\left(\ell_{n}+\frac{3}{2}\right)}\right)
\end{align} 
uniformly in $n$. This follows from the estimate in Appendix \ref{lob}
\begin{align}
|J_{\mu}(z)|\leq \frac{|\frac{z}{2}|^{\mu}e^{|\Im z|}}{\Gamma(\mu+1)} \quad \mu>-\frac{1}{2}
\end{align}
and the derivative identity in the Appendix, \eqref{Ab}. The family $\lambda^{-\ell_{n}}\tilde{j}_{\lambda,Y}(Y_{n})$ is bounded in $L^2_{\loc}(\mathbb{C}^3\setminus 0)$. From Theorem \ref{K2.43} we have that $\Delta_{\rel}\tilde{j}_{\lambda,Y}(Y_{n})=\lambda^2\tilde{j}_{\lambda,Y}(Y_{n})$. This gives that the family is bounded in $H^s_{\loc}(\mathbb{C}^3\setminus 0)$ for $s\in 2\mathbb{N}$. We note that $R_{\lambda}=\mathcal{O}(\lambda^{-2})$ if we consider it as a map $H^s_{\mathrm{com}}(M,\mathbb{R}^3)\rightarrow H_{\mathrm{loc}}^{s+2}(M,\mathbb{R}^3)$. This gives the desired result. 
\end{proof} 
We make the definition
\begin{align}\label{repA}
E_{\lambda}^{\mathrm{curl}}(Y)|_{M}=\tilde{j}_{\lambda,Y}(Y)+\tilde{h}^{(1)}_{\lambda,Y}(A_{\lambda}Y)
\end{align}
for some matrix $A_{\lambda}$. The matrix is designed to match relative boundary conditions, in this case the only missing condition is $\nu\times E_{\lambda}^{\mathrm{curl}}(Y)=0$, as given in Section \ref{miescattering}. Now, we want to relate our divergence free generalized eigenfunction to \eqref{repA}. However before proceeding to establish a relationship to the functional calculus first we need more information about the scattering amplitude. In particular without the computation of $A_{\lambda}$ it is not clear why the representations converge in $C^{\infty}(M,\mathbb{C}^3)$. We let $B_{R_1}(0)$, $B_{R_2}(0)$ denote open balls centered at the origin with radii $R_1$ and $R_2$ respectively, which we abbreviate as $B_1$ and $B_2$. We assume $R_1, R_2$ are sufficiently large that $\mathcal{O}\subset \overline{B_1}$, $B_1\subset B_2$. The ball $B_1$ also must be such that we have $\Delta_{\rel}|_{\mathbb{R}^3\setminus B_1}=\Delta_{\free}|_{\mathbb{R}^3\setminus B_1}$. (This notational setup is adapted from \cite{DZbook}, Setup in Section 4.1 and Definition 4.1.) We claim 
\begin{lem}\label{long}
If $Y\in L^2_t(\mathbb{S}^{2},\mathbb{C}^3)$ then $A_{\lambda}\in C^{\infty}(\mathbb{S}^2,\mathbb{C}^3)$ can be chosen uniquely in \eqref{repA} so that
\begin{enumerate}
\item $\mathrm{div} E_{\lambda}^{\mathrm{curl}}(Y)=0$ in $M$
\item $\nu\times E_{\lambda}^{\mathrm{curl}}(Y) =0$ 
\item $E_{\lambda}^{\mathrm{curl}}(Y)$ solves the equation $(\mathrm{curl}\mathrm{curl}-\lambda^2)u=0$.
\item 
$\varphi E_{\lambda}^{\mathrm{curl}}(Y)\in \mathrm{dom}(\Delta_{\rel})$ for any $\varphi\in C_0^{\infty}(M)$ with $\nabla\varphi=0$ near $\partial\mathcal{O}$. 
 \end{enumerate}
\end{lem}
\begin{proof}
The third and fourth point follow by the definition of the explicit functions involved. Then for any $B_1$, it is possible to select a $\chi\in C_0^{\infty}(M)$ as in the definition of $E_{\lambda}(Y)$ such that if we define 
$$g=E_{\lambda}(Y)-\tilde{j}_{\lambda,Y}(Y)$$ then $g$ is outgoing on $\mathbb{R}^3\setminus B_2$ for $\lambda\neq 0$ and smooth, by definition (c.f Appendix \ref{rellich} for the definition of outgoing). It follows from  Appendix \ref{rellich} that we have $g=\tilde{h}_{\lambda}^{(1)}(A_{\lambda}(Y))$ for a unique $A_{\lambda}(Y)\in C^{\infty}(\mathbb{S}^2,\mathbb{C}^3)$ on $\mathbb{R}^3\setminus B_1$. Uniqueness follows by the asymptotic expansions, as a result of Appendix \ref{rellich}. Indeed we have that if $Y\in L^2_t(\mathbb{S}^2,\mathbb{C}^3)$ is smooth then we have 
\begin{align}
\label{eqnE1}
E_{\lambda}(Y)=\frac{e^{-\rmi\lambda r}e^{\frac{\rmi\pi}{2}}}{r}Y+\frac{e^{\rmi\lambda r}e^{-\frac{\rmi\pi}{2}}}{r}(\tau(Y)+\tau(A_{\lambda}Y))+O\left(\frac{1}{r^2}\right) \quad r\rightarrow \infty
\end{align} 
where $\tau: C^{\infty}(\mathbb{S}^2,\R^3)\rightarrow C^{\infty}(\mathbb{S}^2,\R^3)$, $f(\theta)\mapsto f(-\theta)$ is the pull back of the antipodal map. This is the classical way to find the scatting matrix \cite{melrosebook}, Section 5.1. Uniqueness of the representation follows immediately from Rellich's theorem, c.f. for example both variants Theorems 4.17 and 4.18 in \cite{DZbook}. These theorems in \cite{DZbook} are stated for functions but their proofs generalize trivially to vectors here because the background metric is Euclidean. In other words, for every $Y\in C^{\infty}(\mathbb{S}^2,\R^3)$ there exists a unique $\Psi_{\lambda}\in C^{\infty}(\mathbb{S}^2,\R^3)$ and a unique solution of $E_{\lambda}(Y)$ of $$(\Delta_{\rel}-\lambda^2)E_{\lambda}(Y)=0$$ such that 
\begin{align}\label{eqnE2}
E_{\lambda}(Y)=\frac{e^{-\rmi\lambda r}e^{\frac{\rmi\pi}{2}}}{r}Y+\frac{e^{\rmi\lambda r}e^{-\frac{\rmi\pi}{2}}}{r}\Psi_{\lambda}+O\left(\frac{1}{r^2}\right) \quad r\rightarrow \infty
\end{align}  
by comparison with the above we get $\Psi_{\lambda}=\tau(Y+A_{\lambda}Y)$, c.f. Appendix \ref{rellich}, Prop \ref{OSprop}. Because we know that $\mathrm{div}(E_{\lambda}^{\mathrm{curl}}(Y))=0$ by construction, $E_{\lambda}^{\mathrm{curl}}(Y)$ must also must solve the Helmholtz equation on $M$. The asymptotic structure of $E_{\lambda}^{\mathrm{curl}}(Y)$ is the same as in \eqref{eqnE2}. This is proved in Proposition \ref{OSprop} and Lemma \ref{replacebase} in Appendix \ref{rellich}.  Therefore the representation $E_{\lambda}(Y)$ agrees with $E_{\lambda}^{\mathrm{curl}}(Y)$ on $\mathbb{R}^3\setminus B_1$, for a ball $B_1$ with $\mathcal{O}\subset \overline{B_1}$ following Theorem 4.17 in \cite{DZbook}. Alternatively one can arrange so the difference $E_{\lambda}^{\mathrm{curl}}(Y)-E_{\lambda}(Y)$ is outgoing on $\mathbb{R}^3\setminus B_2$ for any $B_2$ with $B_1\subset B_2$ and no outgoing solutions to the Helmholtz equation exist, c.f. Theorem 4.18 in \cite{DZbook}. The explicit form of $A_{\lambda}$ is found in the next subsection \ref{miescattering}. 
\end{proof} 
Finding the matrix $A_{\lambda}$ is known as Mie scattering where usually this is referred to as the T-matrix in the engineering and physics literature. The problem can also involve mixed boundary conditions for $A$ and $f$ as in \eqref{system} rather than the two boundary conditions in Lemma \ref{long}. Perfectly conducting boundary conditions, however, are the easiest to relate to the functional calculus which is the connection we establish here. 
\begin{lem}\label{equalityofeigenfunctions}
We have that $E_{\lambda}(Y)=E_{\lambda}^{\mathrm{curl}}(Y)$ on $M$.
\end{lem}
\begin{proof}
Lemma \ref{long} shows that the conditions of Proposition 1.2 in \cite{RBM} are satisfied in order, but instead of the relative Laplacian we have the curl operator. We could, of course, use a lengthy calculation to find $E_{\lambda}(Y)$ in this basis. However our $E_{\lambda}^{\mathrm{curl}}(Y)$ are already constructed so that $\mathrm{div}\, E_{\lambda}^{\mathrm{curl}}(Y)=0$ and $Y\in L^2_t(\mathbb{S}^{2},\mathbb{C}^3)$ and therefore they satisfy the Helmholtz equation on $M$. This follows since by construction every term in $E_{\lambda}^{\mathrm{curl}}(Y)$ is of the form of either $\mathrm{curl}(x f_{\ell}(\lambda r)Y_{\ell m})$ or $\mathrm{curl}(xf_{\ell}(\lambda r)Y_{\ell m})$ where $f_{\ell}=j_{\ell}$ or $h^{(1)}_{\ell}$ multiplied by a scalar. Uniqueness of the representation on $M\setminus B_1$ using \eqref{eqnE1} and \eqref{eqnE2} follows immediately from Rellich's theorem, c.f. for example Theorem 4.18 in \cite{DZbook}, as mentioned previously in the proof of Lemma \ref{long}. Since the representations agree on $\partial B_1$ and satisfy the boundary condition $\nu\times (E_{\lambda}^{\mathrm{curl}}(Y)-E(Y))=0$, the representations agree as solutions to the Maxwell spectral problem \eqref{esystem} on the set $B_1\setminus \mathcal{O}$. This set can be made arbitrarily thin anyways due to the fact $\mathcal{O}$ is also a ball. 
\end{proof} 
The major result of this section is then the observation that instead of a representation for the time-harmonic Maxwell equations, Lemmas \ref{long} and \ref{equalityofeigenfunctions} enable us to take this representation one step further and use the functional calculus to represent evolution operators associated to perfectly conducting objects. 
We remark for completeness that as in \cite{OS} that the generalised eigenfunctions $E_{\lambda}(Y)$ can also be considered as distributions in $\lambda$ with values in the Schwartz space of functions $\mathcal{S}(M; \mathbb{C}^3)$. Whenever $g\in C_0^{\infty}(\mathbb{R}^+)$ then $$\int\limits_{\mathbb{R}}g(\lambda)E_{\lambda}(Y)\,d\lambda$$ is square integrable. The inner product 
\begin{align}\label{diraclambda}
\langle E_{\lambda}(Y),E_{\mu}(\tilde{Y})\rangle_{L^2(M)}=(4\pi\lambda)\delta(\lambda^2-\mu^2)\langle Y,\tilde{Y}\rangle_{L^2(\mathbb{S}^{d-1})}
\end{align}
can be seen as a bidistribution in $\mathcal{D}'(\mathbb{R}_+\times\mathbb{R}_+)$. The way this is computed is by taking the limit
\begin{align}
\lim\limits_{R\rightarrow\infty}(\langle \Delta_{\rel}E_{\lambda}(Y),\chi_{R}E_{\mu}(\tilde{Y})\rangle_{L^2(M)}-(\langle E_{\lambda}(Y),\chi_{R}\Delta_{\rel}E_{\mu}(\tilde{Y})\rangle_{L^2(M)}.
\end{align}
This is due directly to the previous two Lemmas. Now we proceed to prove Stone's theorem in this setting. 

\begin{prop}\label{StoneFY}
Let $f,g\in C^\infty_0(M,\mathbb{R}^3)$ and $\lambda>0$. For any bounded Borel function $k: \R \to \C$ 
if $\mathrm{div} f=0=\mathrm{div}g$ then the Stone's formula for divergence free vector fields is given by
\begin{align} \label{deltaf}
(R_\lambda-R_{-\lambda})f=\frac{\rmi}{2\lambda}\sum\limits_n E_{\lambda}(Y_n)\langle f, E_{\lambda}(Y_n)\rangle
\end{align} 
in $C^\infty(M,\mathbb{R}^3)$. This implies
\begin{equation}
\label{Stonecoclosed}
\langle P_{ac} k(\Delta_{\mathrm{rel}}) f,g\rangle
= \frac{1}{2\pi}  \sum\limits_{n}  \int_0^\infty k(\lambda^2) \langle f, E_{\lambda}(Y_n)\rangle \langle  E_{\lambda}(Y_n),g \rangle \,\der \lambda,
\end{equation}
where $P_{ac}$ is the projection onto the absolutely continuous spectrum.
\end{prop}
Stone's theorem as well as its corollaries in the previous section hold by replacing the generalized eigenfunctions written with respect to a basis of harmonic polynomials in \cite{RBM} with $E_{\lambda}(Y_{n})$, using Rellich's theorem.  We the proof include it here for completeness since we have a new eigenbasis. 
\begin{proof}[Proof of Proposition \ref{StoneFY}]
Assume that $f\in C_0^{\infty}(M,\mathbb{R}^3)$ is a divergence free vector field. Then for fixed $\lambda>0$ we have that $R_{-\lambda} f$
is incoming. We also know that incoming waves have asymptotic behavior by Appendix \ref{rellich}
\begin{align}
R_{-\lambda} f=\frac{e^{-\rmi\lambda r}e^{\frac{\rmi\pi}{2}}}{r}\Psi+O\left(\frac{1}{r^2}\right)
\end{align}
as $r\rightarrow \infty$ whenever $\Psi\in C^{\infty}(\mathbb{S}^2,\mathbb{C}^3)$. \\ Therefore we have that 
\begin{align}
(R_{\lambda}-R_{-\lambda})f=\frac{e^{-\rmi\lambda r}e^{\frac{\rmi\pi}{2}}}{r}\Psi+\frac{e^{\rmi\lambda r}e^{-\frac{\rmi\pi}{2}}}{r}\tilde{\Psi}+O\left(\frac{1}{r^2}\right)
\end{align}
as $r\rightarrow \infty$ for some $\tilde{\Psi}\in C^{\infty}(\mathbb{S}^2,\mathbb{C}^3)$. \\Because $E_{\lambda}(\Psi)$ and  $(R_{\lambda}-R_{-\lambda})f$ have the same asymptotic term then they necessarily coincide by Rellich's uniqueness theorem. Integration by parts results in 
\begin{align}
\langle f,E_{\lambda}(Y_{n})\rangle_{L^2(M)}=\langle(\Delta-\lambda^2)R_{-\lambda}f, E_{\lambda}(Y_{n})\rangle_{L^2(M)}=-2\rmi\lambda\langle \Psi,Y_{n}\rangle_{L^2(\mathbb{S}^2)}.
\end{align}
Expanding $\Psi$ into the basis of vector spherical harmonics we see that 
\begin{align}
\Psi=\sum\limits_{n}\langle \Psi,Y_{n}\rangle_{L^2(M)}Y_{n}
\end{align}
with the sum converging in $C^{\infty}(\mathbb{S}^2,\mathbb{C}^3)$ because $\Psi$ is smooth. Therefore we have that 
\begin{align}
\Psi=\frac{i}{2\lambda}\sum\limits_{n}\langle f, E_{\lambda}(Y_{n})\rangle_{L^2(M)}Y_{n}.
\end{align}
It follows that 
\begin{align}
(R_{\lambda}-R_{-\lambda})f=E_{\lambda}(\Psi)=\frac{\rmi}{2\lambda}\sum\limits_{n}E_{\lambda}(Y_{n})\langle f,E_{\lambda}(Y_{n})\rangle_{L^2(M)}.
\end{align} 
From the properties of the resolvent and $E_{\lambda}(Y)$, the map $Y\mapsto E_{\lambda}(Y)$ for $Y\in L^2_t(\mathbb{S}^2,\mathbb{C}^3)$ is continuous. The sum converges in $C^{\infty}(M,\mathbb{C}^3)$ due to the construction in Lemma \ref{maxgen} and the bounds there. The representation formula follows. One can compute using the Bessel and Hankel function asymptotics that the equalities still hold for $\Im \lambda>0$ uniformly on compact subsets of the complex plane. Combining these observations, Stone's theorem then gives the spectral measure and the complete spectral decomposition. The last equality \eqref{Stonecoclosed} is a direct result of Stone's theorem once the formula for the spectral measure has been established. For any spectral measure $dB_{\lambda}$ on the real line corresponding to the continuous spectrum we have that for any $g,f\in C_0^{\infty}(M;\mathbb{R}^3)$ 
\begin{align}
\langle dB_{\lambda}f,g\rangle=\frac{1}{2\pi}\chi_{[0,\infty)}(\lambda)\sum\limits_{n}\langle f,E_{\lambda}(Y_n)\rangle \langle E_{\lambda}(Y_n),g\rangle\,d\lambda.
\end{align}
Therefore for any bounded Borel function $k$ we can conclude \eqref{Stonecoclosed}.
\end{proof} 
We then have
\begin{prop}\label{stone2Y}
 If $k$ is a Borel function with $k= O((1+\lambda^2)^{-q})$ for all $q\in\mathbb{N}$ we have that $k(\Delta_{\mathrm{rel}})$ has smooth integral kernel
 $K_k \in C_0^{\infty}(M \times M;\mathbb{C}^3\times \mathbb{C}^3 )$. Its domain is the divergence free $f\in C_0^{\infty}(M;\mathbb{R}^3)$ and 
 \begin{gather}
K_k(x,y) = k(0) \sum_{j=1}^N u_j(x) \otimes (u_j(y))^* + \frac{1}{2\pi}  \sum\limits_{n}  \int_0^\infty k(\lambda^2) E_{\lambda}(Y_n)(x) \otimes E_{\lambda}(Y_n)^*(y)\,\der \lambda,
\end{gather}
where the sum converges in $C^{\infty}_0(M \times M; \mathbb{C}^3 \times \mathbb{C}^3)$. The $u_j's$ are the $L^2$ normalized functions in $\mathrm{ker}(\Delta_{\mathrm{rel}})$. 
\end{prop}
The proof of Proposition \ref{stone2Y} follows basically identically to Proposition \ref{StoneFY} and therefore will not be repeated here, c.f. for example Prop. 2.8 in \cite{RBM} with $\Phi$ replaced by $Y$. The proof there is formally the same. 

\subsection{Mie Scattering for the sphere}\label{miescattering}
 We know that 
\begin{equation*}
E_{Y}^{\mathrm{curl}}(Y_\mu)|_{\R^3\backslash \cal{O}}=\tilde{j}_\lambda(Y_\mu)+\tilde{h}_\lambda^{(1)}(A_{\lambda} Y_\mu)=\tilde{h}_\lambda^{(2)}(Y_\mu)+\tilde{h}_\lambda^{(1)}(S_{\lambda}Y_\mu).
\end{equation*}
by construction for some unique matrices $A_{\lambda}$ and $S_{\lambda}$, as a result of Lemmas \ref{long} and \ref{equalityofeigenfunctions}. It remains to find them explicitly by matching up the boundary conditions on the sphere. Furthermore by symmetry one has $A_{\lambda}=A_{\lambda}^t$ on $Y\in L^2_t(\mathbb{S}^2,\mathbb{C}^3)$. This formula is exact in the exterior of a sphere in 3 dimensions. In any dimension, representations for generalized eigenfunctions using Mie scattering could be done in the basis of harmonic polynomials, used in \cite{OS,RBM} but their corresponding scattering matrix is not diagonal which is why we chose to switch to this basis. This will allow us to do calculations with the kernel defined in Proposition \ref{stone2Y} explicitly for the Maxwell propagator. We start with the following Lemma
\begin{lem}\label{maxgen}
The Maxwell generalized eigenfunction for a perfectly conducting sphere of radius $\rho$ is given by 
\begin{equation}
\label{generalisedE}
E_{\lambda}^{\mathrm{curl}}(Y_\mu)|_{\R^3\backslash \cal{O}}=\tilde{j}_\lambda(Y_\mu)+\tilde{h}_\lambda^{(1)}(A^t_{\lambda}Y_\mu)
\end{equation}
where $A^t_{\lambda}$ is a diagonal matrix given by 
\begin{align}\label{Atan}
\langle A^t_{\lambda} \Psi_{1,\ell,m},\Psi_{1,\ell,m} \rangle= -\frac{2j_{\ell}(\lambda\rho)}{h_{\ell}^{(1)}(\lambda\rho)},
\end{align}
and
\begin{align}\label{Anormal}
\langle A^t_{\lambda} \Psi_{2,\ell,m},\Psi_{2,\ell,m} \rangle= -\frac{2(zj_{\ell}(z))'|_{z=\lambda \rho}}{(zh_{\ell}^{(1)}(z))'|_{z=\lambda \rho}},
 \end{align}
where $\Psi_{j,\ell,m}$ $j=1,2$ form a basis for $L^2_t(\mathbb{S}^{2},\mathbb{C}^3)$. 
\end{lem} 
\begin{proof}
It suffices to find $A_{\lambda}$ so that $\nu\times (\tilde{j}_\lambda(Y_\mu)+\tilde{h}_\lambda^{(1)}(A_{\lambda}Y_\mu))=0.$ 
In the case of the obstacle being a sphere of radius $\rho$, this implies the tangential components must vanish:
\begin{equation*}
\left\{
\begin{aligned}
&\text{For $\mu=(1,\ell,m)$}, \quad 2\lambda\Psi_{\mu}(\theta)j_{\ell}(\lambda \rho)(-\rmi )^{\ell}+\lambda\sum\limits_{n=(j,\ell_0,m_0)} (A_{\lambda})_{\mu,n}\Psi_{j,\ell_0,m_0}(\theta) h^{(1)}_{\ell_0}(\lambda \rho)(-\rmi)^{\ell_0}=0. \\
&\text{For $\mu=(2,\ell,m)$}, \quad 2\Psi_{\mu}\frac{\left (z j_{\ell}(z)\right)'}{\rho}|_{z=\lambda \rho}(-\rmi)^{\ell-1}\\
&\qquad\qquad\qquad\qquad\qquad\qquad\quad+\sum\limits_{n=(j,\ell_0,m_0)} (A_{\lambda})_{\mu,n}\Psi_{j,\ell_0,m_0}(\theta) \frac{\left(z h^{(1)}_{\ell}(z)\right)'}{\rho}|_{z=\lambda \rho}(-\rmi)^{\ell_0-1}=0.
\end{aligned}
\right.
\end{equation*}
Since the vector spherical harmonics with $j=1,2$ form a basis for the space $L^2_t(\mathbb{S}^{2},\mathbb{C}^3)$, the tangential part is given by the diagonal matrix with two different types of diagonal entries. 
As a consequence, $(A_\lambda)_{j,\ell,m}$ is diagonal and its components are given by \eqref{Atan} and \eqref{Anormal}. This result is popular in the physics literature \cite{scat,mie,stratton}. This concludes the result for $A^t_{\lambda}$. The result for $S^t_{\lambda}$ follows immediately from the definition.
\end{proof}
Applying Proposition \ref{stone2Y} and Proposition \ref{StoneFY} to $k(\lambda^2)=\varphi(h \lambda)e^{\rmi \lambda t}$, we have the kernel of $k(\Delta_{\mathrm{rel}})$ on divergence free vector fields is found by using \eqref{stone2Y} and also Lemma \ref{equalityofeigenfunctions}. We label this as $K^{\mathcal{O},m}(y,y',\rho,t,h)$, 
\begin{align}\label{maxkernel}
&K^{\mathcal{O},m}(y,y',\rho,t,h) \nonumber
:= \\& \frac{1}{2\pi h}  \sum\limits_{\mu}  \int_0^\infty \varphi(\lambda)e^{\rmi \frac{\lambda}{h} t}\left(E_{\frac{\lambda}{h}}^{\mathrm{curl}}(Y_\mu)(y)\right)\left(E_{\frac{\lambda}{h}}^{\mathrm{curl}}(Y_\mu)(y')\right)^* \der \lambda.
\end{align}
We may use this as a representation of our operator on $f\in C_0^{\infty}(M; \mathbb{R}^3)$ which are divergence free from Proposition \ref{stone2Y} followed by Lemma \ref{equalityofeigenfunctions}. We bound the representation in spherical coordinates, identifying $y,y'$ with $(r,\theta)$ and $(r',\theta')$ respectively. 
 
We bound this in the next section.  
\section{Proof of Theorem \ref{spheretheorem2}}\label{laterproofs}
We suppress the dimension $d=3$ for this section where it is understood. We define for $r\geq \rho$ the elements of the diagonal matrices $B_{1,\lambda}(\rho,r),B_{2,\lambda}(\rho,r),$ and $B_{3,\lambda}(\rho,r)$ as
\begin{align*}
B_{1,\ell,\lambda}(\rho,r):=-\frac{h^{(1)}_{\ell}(\lambda r)}{h^{(1)}_{\ell}(\lambda \rho)}
\end{align*}
and  
\begin{align}
B_{2,\ell,\lambda}(\rho,r):=-\frac{\frac{1}{r}(zh^{(1)}_{\ell}(z))'|_{z=\lambda r}}{(zh^{(1)}_{\ell}(z))'|_{_{z=\lambda \rho}}}
\end{align}
and lastly
\begin{align}
B_{3,\ell,\lambda}(\rho,r):=-\frac{\sqrt{\ell(\ell+1)}}{r}\frac{h^{(1)}_\ell(\lambda r)}{(zh^{(1)}_{\ell}(z))'|_{z=\lambda \rho}}
\end{align}
respectively.
We then have the following Lemma 
\begin{lem}\label{matrixbound}
The following estimates hold for the entries of $B_{1,\ell,\frac{\lambda}{h}}(\rho,r)$ for $\lambda\in \mathrm{supp}(\varphi)$
\begin{align}\label{Bbound}
|B_{1,\ell,\frac{\lambda}{h}}(\rho,r)|\leq 1 \qquad \left|\frac{\der}{\der\lambda}B_{1,\ell,\frac{\lambda}{h}}(\rho,r)\right|\leq \frac{2}{h}\max\{r,(\ell+1)h\}
\end{align}
and for the entries of $B_{2,\ell,\frac{\lambda}{h}}(\rho,r)$ there is a constant independent of $r$ and $\ell$ such that for $w=\frac{\lambda\rho}{h}$
\begin{align}\label{Bbound2}
|B_{2,\ell,\frac{\lambda}{h}}(\rho,r)|\leq C\begin{cases}\frac{1}{h} \quad \ell<we \\ \frac{\ell h}{\lambda}+1 \quad \ell >we\end{cases} \qquad \left|\frac{\der}{\der\lambda}B_{2,\ell,\frac{\lambda}{h}}(\rho,r)\right|\leq  C\begin{cases}\frac{r}{h^2} \quad \ell<we \\ \frac{\ell(\ell+1)h}{\lambda}+1  \quad \ell >we\end{cases}
\end{align}
and for the entries of $B_{3,\ell,\frac{\lambda}{h}}(\rho,r)$ there is a constant independent of $r$ and $\ell$ such that for $w=\frac{\lambda\rho}{h}$
\begin{align}\label{Bbound3}
|B_{3,\ell,\frac{\lambda}{h}}(\rho,r)|\leq C\begin{cases}\frac{1}{h} \quad \ell<we \\ \frac{\ell h}{\lambda}\quad \ell >we\end{cases} \qquad \left|\frac{\der}{\der\lambda}B_{3,\ell,\frac{\lambda}{h}}(\rho,r)\right|\leq  C\begin{cases}\frac{r}{h^2} \quad \ell<we \\ \frac{\ell(\ell+1)h}{\lambda}+1  \quad \ell >we\end{cases}.
\end{align}
\end{lem}
\begin{proof}
 We have from the Appendix, equation \eqref{lob1}, that the first bound in \eqref{Bbound} holds. By which we mean, this is directly seen from the fact the absolute value of the spherical Hankel function is a decreasing function of the $r$ variable.  We use \eqref{recurrence} to obtain that 
\begin{equation}\label{der}
\frac{\der}{\der w}h^{(1)}_{\ell}(w)=h^{(1)}_{\ell-1}(w)-\left(\frac{\ell+1}{w}\right)h_{\ell}^{(1)}(w)=
-h^{(1)}_{\ell+1}(w)+\left(\frac{\ell}{w}\right)h_{\ell}^{(1)}(w)
\end{equation}
along with change of variables to get the second bound in \eqref{Bbound} combined with \eqref{lob1}.

For the terms in \eqref{Bbound2}, we start by recalling the following identities for the derivatives of the spherical Hankel functions in dimension 3. 
Using \eqref{der} we obtain 
\begin{equation}\label{1a}
\frac{\der}{\der w}(w h_{\ell}^{(1)}(w))=
 wh^{(1)}_{\ell-1}(w)-\ell h_{\ell}^{(1)}(w) =-wh_{\ell+1}^{(1)}(w)+(\ell+1)h_{\ell}^{(1)}(w)
\end{equation}
and using the equation for the spherical Hankel function itself 
\begin{align}\label{1b}
\frac{\der^2}{\der w^2}(w h_{\ell}^{(1)}(w))=-wh^{(1)}_{\ell}(w)+\frac{\ell(\ell+1)}{w}h_{\ell}^{(1)}(w). 
\end{align}
The numerators can then be bounded by using \eqref{1a} iteratively. 
If 
\begin{align}
\ell+1<\frac{\lambda \rho}{2h} \quad \lambda\in \mathrm{supp}(\varphi)
\end{align} 
 then we get for $w=\frac{\lambda \rho}{h}$
\begin{align}
|wh_{\ell+1}^{(1)}(w)-(\ell+1)h_{\ell}^{(1)}(w)| \geq |wh_{\ell+1}^{(1)}(w)|-(\ell+1)|h_{\ell}^{(1)}(w)|\geq \frac{|wh_{\ell+1}^{(1)}(w)|}{2}.
\end{align} 
We have that for $\ell+1< \frac{\lambda \rho}{2h}$
\begin{equation}
\label{B2eqn1}
|B_{2,\ell,\frac{\lambda}{h}}(\rho,r)|\leq \frac{4}{|\frac{\lambda\rho}{h}h_{\ell+1}^{(1)}(\frac{\lambda\rho}{h})|}\left|\frac{\lambda}{h}h_{\ell+1}^{(1)}\left(\frac{\lambda r}{h}\right)\right|\leq 4.
\end{equation} 
Now to derive the first inequality in \eqref{Bbound2} we have to continue to look for a lower bound for 
\begin{align}
|wh^{(1)}_{\ell-1}(w)-\ell h_{\ell}^{(1)}(w)|^{-1}.
\end{align}
in terms of various orders of $\ell$ with respect to $w$. If $ew<\ell$, then 
\begin{equation}
\label{B2eqn2}
|wh^{(1)}_{\ell-1}(w)-\ell h_{\ell}^{(1)}(w)|\geq 
\left(\frac{\ell}{w}-1\right)|wh_{\ell}^{(1)}(w)|\geq(e-1)|wh_{\ell}^{(1)}(w)|
\end{equation}
and we get a bound of 
\begin{align}
|wh^{(1)}_{\ell-1}(w)-\ell h_{\ell}^{(1)}(w)|^{-1}\leq \frac{1}{e-1}\left|wh_{\ell}^{(1)}(w)\right|^{-1}.
\end{align}
This bound then results in 
\begin{align}
|B_{2,\ell,\frac{\lambda}{h}}(\rho,r)|\leq \frac{\ell h+\lambda}{\lambda \rho(e-1)}.
\end{align}
The difficult region is then when $\frac{\ell}{e}<w<2(\ell+1)$ whenever $w^{-1}$ is sufficiently small, e.g. $w=\frac{\lambda \rho}{h}$. 
Using equation \eqref{lob1} from the Appendix section \ref{Abf} we have that 
\begin{equation*}
2\left|\frac{d}{dw}(wh_\ell^{(1)}(w))\right||wh_\ell^{(1)}(w)|\geq \left|\frac{\der}{\der w}\left(\sum\limits_{k=0}^\ell\frac{s_k(\ell+\frac{1}{2})}{w^{2k}}\right)\right|
\geq \sum\limits_{k=0}^\ell\frac{2ks_k(\ell+\frac{1}{2})}{w^{2k+1}}
\geq \frac{|wh_\ell^{(1)}(w)|^2-1}{w},
\end{equation*}
which implies
\begin{equation*}
\left|\left(\frac{d}{dw}(wh_\ell^{(1)}(w))\right)\right|^{-1}\leq \frac{2w|wh_\ell^{(1)}(w)|}{|wh_\ell^{(1)}(w)|^2-1}
\end{equation*}
with $w=\frac{\lambda \rho}{h}$. We also have for $\ell<\frac{\lambda \rho e}{h}$ from the identity \eqref{1a}
\begin{align}
\left|\frac{d}{dw}\left(wh_{\ell}^{(1)}(w)\right)|_{w=\frac{\lambda r}{h}}\right|\leq 4\left|\frac{\lambda\rho}{h}h_{\ell}^{(1)}\left(\frac{\lambda \rho}{h}\right)\right|.
\end{align}
This gives 
\begin{align*}
&|B_{2,\ell,\frac{\lambda}{h}}(\rho,r)|\leq 
\left.\left(\frac{2w|wh_{\ell}^{(1)}(w)|}{r(|wh_{\ell}^{(1)}(w)|^2-1)}\right)\right|_{w=\frac{\lambda \rho}{h}}\left|\frac{d}{d\tilde{x}}(\tilde{x}h_{\ell}^{(1)}(\tilde{x}))|_{\tilde{x}=\frac{\lambda r}{h}}\right|\leq \\&
\left.\left(\frac{8w|wh_{\ell}^{(1)}(w)|^2}{|wh_{\ell}^{(1)}(w)|^2-1}\right)\right|_{w=\frac{\lambda \rho}{h}}=
8w\left.\left(1+\frac{1}{|wh _\ell^{(1)}(w)|^2-1}\right)\right|_{w=\frac{\lambda \rho}{h}}.
\end{align*}
Because we have that 
\begin{align*}
\left|wh_\ell^{(1)}\left(w\right)\right|^2\geq 1+\frac{\ell(\ell+1)}{w^2}\geq 1+\frac{\ell}{2w}\geq 1+\frac{1}{4}-\frac{1}{2w},
\end{align*} 
it follows that if $h\leq \frac{1}{4}$ we have $|B_{2,\ell,\frac{\lambda}{h}}(\rho,r)|\le Ch^{-1}$.
This bound is certainly not optimal. The derivative with respect to $\lambda$ of $B_{2,\ell,\frac{\lambda}{h}}$ follows by using \eqref{1b} and the previously established lower bounds. The bound for $|B_{2,\ell,\frac{\lambda}{h}}(\rho,r)|$ then follows. 

Similar to the proof of \eqref{B2eqn1}, we have that for $\ell+1< \frac{\lambda \rho}{2h}$
\begin{equation*}
|B_{3,\ell,\frac{\lambda}{h}}(\rho,r)|\leq \frac{2\ell}{r}\frac{h^{(1)}_\ell(\frac{\lambda r}{h})}{|\frac{\lambda\rho}{h}h_{\ell+1}^{(1)}(\frac{\lambda\rho}{h})|}\leq 2.
\end{equation*} 
For $ew<\ell$, the same estimate of \eqref{B2eqn2} gives
\begin{equation*}
|B_{3,\ell,\frac{\lambda}{h}}(\rho,r)|\leq \frac{2\ell}{r}\frac{h^{(1)}_\ell(\frac{\lambda r}{h})}{|\frac{\lambda\rho}{h}h_{\ell}^{(1)}(\frac{\lambda\rho}{h})|}\leq \frac{2\ell h}{\lambda}.
\end{equation*}
When $\frac{\ell}{e}<w<2(\ell+1)$, one has
\begin{equation*}
|B_{3,\ell,\frac{\lambda}{h}}(\rho,r)|\leq \frac{2\ell}{r}\left.\left(\frac{2|wh_{\ell}^{(1)}(w)|^2}{r(|wh_{\ell}^{(1)}(w)|^2-1)}\right)\right|_{w=\frac{\lambda \rho}{h}}\leq \frac{C}{h}.
\end{equation*}
The bounds derivative of $B_{3,\ell,\frac{\lambda}{h}}$ follow similarly. The impediment to obtaining a better estimate in \eqref{Bbound2} is the lack of precise asymptotics for the Hankel function derivatives in this setting.
\end{proof}
  
\begin{proof}[Proof of Theorem \ref{spheretheorem2}]
We start with the most difficult terms occurring in the kernel to analyze 
\begin{align}\label{impcross}
\sum\limits_{n} \int\limits_0^{\infty}\varphi(h\lambda)e^{\rmi t\lambda} \tilde{h}^{(1)}_{\lambda}(A_{\lambda}^tY_n)(r,\theta)\overline{\tilde{h}^{(1)}_{\lambda}(A_{\lambda}^tY_n)(r',\theta')}\,\der\lambda.
\end{align}
By symmetry bounds for the other terms in the kernel
will follow almost immediately, once we have obtained a bound for the kernel \eqref{impcross}. 
%

From \eqref{Atan} when $Y_{n}=\Psi_{1,\ell,m}$ we have that 
\begin{align}\label{outmove}
&\tilde{h}^{(1)}_{\lambda}(A_{\lambda}^tY_n)(r,\theta)= \lambda (A_{\lambda}^t)_{1,\ell,m}\Psi_{1,\ell,m}(\theta)h_{\ell}^{(1)}(\lambda r)(-\rmi)^{\ell}\\
&=2\lambda\Psi_{1,\ell,m}(\theta)B_{1,\ell,\lambda}(\rho,r)j_{\ell}(\lambda \rho)(-\rmi)^{\ell}=\tilde{j}_{\lambda}(B_{1,\lambda}Y_n)(\rho,\theta). 
 \nonumber
\end{align}
By the Cauchy-Schwarz inequality along with \eqref{Bbound} and \eqref{outmove} we obtain when $Y_{n}=\Psi_{1,\ell,m}$ 
\begin{align}\label{diag}
&\left|\int\limits_0^{\infty}\varphi(h\lambda)e^{\rmi t\lambda} \tilde{h}^{(1)}_{\lambda}(A_{\lambda}^tY_n)(r,\theta)\overline{\tilde{h}^{(1)}_{\lambda}(A_{\lambda}^tY_n)(r',\theta')}\,\der\lambda \right|\leq \nonumber \\&
\frac{1}{2}\left|\int\limits_0^{\infty}\varphi(h\lambda)|\tilde{h}^{(1)}_{\lambda}(A_{\lambda}^tY_n)(r,\theta)|^2+|\overline{\tilde{h}^{(1)}_{\lambda}(A_{\lambda}^tY_n)(r',\theta')}|^2\,\der\lambda \right|\leq \nonumber \\&
 \nonumber \frac{1}{2}\int\limits_0^{\infty}\varphi(h\lambda)\lambda^{2} \left(|B_{1,\ell,\lambda}(\rho,r)|^2+|B_{1,\ell,\lambda}(\rho,r')|^2 \right)  \left|Y_{n}(\theta)j_{\ell}(\lambda\rho)(-\rmi)^{\ell}\right|^2\,\der\lambda\leq \nonumber \\&
\int\limits_0^{\infty}\varphi(h\lambda)\lambda^{2}\left|\Psi_{1,\ell,m}(\theta)j_{\ell}(\lambda\rho)(-\rmi)^{\ell}\right|^2\,\der\lambda. 
\end{align}
Let $K^{\Delta_{free}}(y,y',t,h)$ denote the kernel of the free space Maxwell propagator: $\varphi(hD_t)e^{it\sqrt{\Delta_{free}}}$ on divergence free forms. Elementary arguments dictate this has the same $L^{1}\rightarrow L^{\infty}$ operator bound as the free space Dirichlet kernel, c.f. \cite{BCDbook}.  

Summing over $n=(1,\ell,m)$, it follows that \eqref{diag} is bounded by 
\begin{align}
\|K^{\Delta_{free}}(\cdot,\cdot,t,h)\|_{L^{\infty}}\leq \left|K^{\Delta_{free}}(y,y,0,h)|_{|y|=\rho}\right|\leq \frac{C}{h^3}
\end{align}
where the subscript $1$ denotes only the kernel with $\Psi_{1,\ell,m}$ present. The cross terms are bounded analogously. 

The case when $Y_{n}=\Psi_{2,\ell,m}$ first coefficient follows by replacing $B_{1,\ell,\lambda}$ with $B_{2,\ell,\lambda}$ and 
\begin{align}
j_{\ell}\left(\frac{\lambda \rho}{h}\right)\quad \mathrm{with} \quad \left.\frac{(z j_{\ell}(z))'}{z}\right|_{z=\frac{\lambda \rho}{h}}.
\end{align}
Unfortunately because of the different scales of $\ell$ in \eqref{Bbound2} we have to separate this into two cases, one with $\frac{\lambda \rho e}{h}\leq \ell$ and the other with $\frac{\lambda \rho e}{h}\geq \ell$. The first case follows from the previous Cauchy-Schwarz estimate and the first equation in \eqref{Bbound2}. We flip the sums and the integrals where necessary due to Fubini's theorem and use Proposition \ref{StoneFY} combined with Proposition \ref{cvim}. Let $\chi_R$ be a $C^{\infty}(M)$ function which is 1 on a ball of radius $2R$ and $0$ outside a ball of radius $3R$. Indeed, we have that writing $\lambda_0=\frac{\lambda}{h}$ for space reasons: 
\begin{align}\label{diag2}
&\frac{1}{h}\left|\int\limits_0^{\infty}\sum\limits_{\substack{n,\\ \ell\leq \lambda_0 \rho e}}\chi_{R}(r)\chi_R(r')\varphi(\lambda)e^{\rmi t\lambda_0} \tilde{h}^{(1)}_{\lambda_0}(A_{\lambda_0,Y}^tY_n)(r,\theta)\overline{\tilde{h}^{(1)}_{\lambda_0}(A_{\lambda_0,Y}^tY_n)(r',\theta')}\,\der\lambda \right|\leq \nonumber \\&
\frac{1}{2h}\left|\int\limits_0^{\infty}\sum\limits_{\substack{n,\\ \ell\leq \lambda_0\rho e}}\chi_{R}(r)\chi_R(r')\varphi(\lambda)\left(|\tilde{h}^{(1)}_{\lambda_0}(A_{\lambda_0,Y}^tY_n)(r,\theta)|^2+|\overline{\tilde{h}^{(1)}_{\lambda_0}(A_{\lambda_0,Y}^tY_n)(r',\theta')}|^2\right)\,\der\lambda \right|.
\end{align}
For the first terms in \eqref{diag2} we have that 
\begin{align}
&\frac{1}{2h}\int\limits_0^{\infty}\sum\limits_{\substack{\ell,m,\\ \ell\leq \lambda_0 \rho e}} \chi_{R}(r)\chi_R(r')\varphi(\lambda)\\&\times \left(|B_{2,\ell,\lambda_0}(\rho,r)|^2+|B_{2,\ell,\lambda_0}(\rho,r')|^2 \right)\nonumber
 \left|\Psi_{2,\ell,m}(\theta)\frac{(zj_\ell)'}{\rho}|_{z=\lambda_0\rho}(-\rmi)^{\ell-1}\right|^2\,\der\lambda\nonumber \\
\leq &\frac{C}{h^3}\int\limits_0^{\infty}\sum\limits_{\substack{\ell,m,\\ \ell\leq \lambda_0\rho e}} \chi_{R}(r)\chi_R(r')\varphi(\lambda)\left|\Psi_{2,\ell,m}(\theta)\frac{(zj_\ell)'}{\rho}|_{z=\lambda_0\rho}(-\rmi)^{\ell-1}\right|^2\,\der\lambda \nonumber \\
\leq& \frac{C}{h^3}\int\limits_0^{\infty}\sum\limits_{\substack{\ell,m,\\ \ell\leq \lambda_0\rho e}}
\chi_{R}(r)\chi_R(r')\varphi(\lambda)\left|\Psi_{2,\ell,m}(\theta)\frac{zj_{\ell-1}(z)-\ell j_{\ell}(z)}{\rho}|_{z=\lambda_0\rho}(-\rmi)^{\ell-1}\right|^2\,\der\lambda \nonumber \\
\leq&  
\frac{|K^{\Delta_{free}}(y,y,0,h)|_{|y|=\rho}|}{h^2}\leq \frac{C}{h^5}.
\end{align}
For the second terms we obtain 
\begin{align}
&\frac{1}{2h}\int\limits_0^{\infty}\sum\limits_{\substack{\ell,m,\\ \ell\leq \lambda_0 \rho e}} \chi_{R}(r)\chi_R(r')\varphi(\lambda)\\& \nonumber \times \left(|B_{3,\ell,\lambda_0}(\rho,r)|^2+|B_{3,\ell,\lambda_0}(\rho,r')|^2 \right)\nonumber
 \left|\Psi_{3,\ell,m}(\theta)(zj_\ell)'|_{z=\lambda_0 \rho}(-\rmi)^{\ell-1}
\right|^2\,\der\lambda \nonumber \\
\leq &\frac{C}{h^3}\int\limits_0^{\infty}\sum\limits_{\substack{\ell,m,\\ \ell\leq \lambda_0\rho e}} \chi_{R}(r)\chi_R(r')\varphi(\lambda)\left|\Psi_{3,\ell,m}(\theta)(zj_\ell)'|_{z=\lambda_0 \rho}(-\rmi)^{\ell-1} \nonumber
\right|^2\,\der\lambda  \\
\nonumber \leq& \frac{1}{h^2}\|K^{\Delta_{free}}(y,y,0,h)|_{|y|=\rho}\|_{L^{\infty}}\leq \frac{C}{h^5}. \nonumber
\end{align}
For the higher order $\ell$ we have that 
\begin{align}\label{onebig}
&\left|J_{\ell}\left(\frac{\lambda \rho}{h}\right)\right|\leq \frac{1}{\sqrt{2\pi\ell}}\left(\frac{\lambda \rho e}{2h\ell}\right)^{\ell}
&\text{and}&
&\left|j_{\ell}\left(\frac{\lambda \rho}{h}\right)\right|\leq C\ell^{-1}\left(\frac{\lambda \rho e}{2h\ell}\right)^{\ell}.
\end{align}
Similarly,
\begin{align}\label{twobig}
\left|H^{(1)}_{\ell}\left(\frac{\lambda r}{h}\right)\right|\leq 2\sqrt{\frac{2}{\pi\ell}}\left(\frac{\lambda r e}{2h\ell}\right)^{-\ell}.
\end{align}
We then examine the first terms in 
\begin{align}\label{diag3}
&\frac{1}{h}\left|\int\limits_0^{\infty}\sum\limits_{\substack{n,\\ \ell> \lambda_0 \rho e}}\varphi(\lambda)e^{\rmi t\lambda_0} \tilde{h}^{(1)}_{\lambda_0}(A_{\lambda_0,Y}^tY_n)(r,\theta)\overline{\tilde{h}^{(1)}_{\lambda_0}(A_{\lambda_0,Y}^tY_n)(r',\theta')}\,\der\lambda \right|.
\end{align}
We then have from Cauchy-Schwarz and \eqref{Bbound2}
\begin{align*}
&\frac{1}{h}\int\limits_0^{\infty}\sum\limits_{\substack{(\ell,m),\\ \ell>\lambda_0\rho e}}\chi_{R}(r)\chi_R(r')\varphi(\lambda)\left(\frac{\ell h}{\lambda}+1\right)^2\left|\Psi_{2,\ell,m}(\theta)\frac{(zj_{\ell})'}{\rho}|_{z=\lambda_0\rho}(-\rmi)^{\ell}\right|^2\,\der\lambda 
\\
\nonumber\leq & 
Ch\int\limits_0^{\infty}\sum\limits_{\substack{(\ell,m),\\ \ell>\lambda_0\rho e}}\chi_{R}(r)\chi_R(r')\varphi(\lambda)\ell^3\left|\left.\left(zj_\ell(z)-\ell j_\ell(z)\right)\right|_{z=\lambda_0\rho}\right|^2\,\der\lambda 
\\
\nonumber\leq & 
\frac{C}{h}\int\limits_0^{\infty}\chi_{R}(r)\chi_R(r')\varphi(\lambda)\sum\limits_{\ell\ge 1}\frac{\ell^5}{4^\ell}\,\der\lambda\le \frac{C}{h}\chi_{R}(r)\chi_R(r'),
\end{align*}
where the above constants only depend on $\rho$ and $\phi$. In the second bound we used $\|Y_{n}\|_{L^{\infty}}\leq C\sqrt{\ell}$ since as remarked earlier $Y_{n}$ can have degree $\ell-1,\ell$ or $\ell+1$. In the last bound, we have used \eqref{onebig}. Using equation \eqref{1a}, this allows us to conclude a bound of $Ch^{-5}$ because of the extra factor of $\ell$ in the first equation of \eqref{Bbound2}. Notice that $\ell$ starts at $1$ for the Maxwell equations. The bound for the second terms in \eqref{diag2} and \eqref{diag3} involving $\Psi_{3,\ell,m}$ follows analogously. 

For the case $t>h$ we separate our investigation into three cases.  
\begin{enumerate} 
\item Case $|\lambda r|<h$, this case is empty by assumption on $a$, in the main theorem- $\rho a\geq 2$ 
\item Case $|\lambda r|>h$, we consider $\ell<\frac{\lambda \rho e}{h}$, where we have assumed $r\geq r'$ without loss of generality. Let the subscript $-$ denote the restriction to this case.  
We examine the difficult terms of the form
\begin{align}\label{cross1}
K_{-}(y,y',t,h)=&\frac{\chi_{R}(r)\chi_R(r')}{2\pi h} \\
& \nonumber \times \int_0^\infty \sum\limits_{n_{-}} \varphi(\lambda)e^{\rmi \frac{\lambda}{h} t} \left(\tilde{h}_{\frac{\lambda}{h}}^{(2)}\left(Y_n\right)(r,\theta)\right)\left(\tilde{h}^{(1)}_{\frac{\lambda}{h}}\left(S_{\frac{\lambda}{h}}^tY_n\right)(r',\theta')\right)^*\der \lambda.
\end{align}
The extra decay comes from integrating by parts once. We are interested in the $L^\infty$-bound of $\left|\frac{t}{h} K_{-}\right|$. Note that $\frac{t}{h}e^{\rmi \frac{\lambda}{h} t}$ can be replaced by $-\rmi\frac{\der}{\der \lambda}e^{\rmi \frac{\lambda}{h} t}$ in the integrand of \eqref{cross1}. The derivative of the spherical Bessel function follows the same relationship as \eqref{der}. We have that 
\begin{equation}\label{dervbig}
\frac{\der}{\der\lambda}\left(h^{(1)}_{\ell}\left(\frac{\lambda r}{h}\right)A_{\frac{\lambda}{h},Y,\ell}\right)=
\frac{\der}{\der\lambda}B_{1,\ell,\frac{\lambda}{h}}(\rho,r)\left(j_{\ell}\left(\frac{\lambda \rho}{h}\right)\right)+B_{1,\ell,\frac{\lambda}{h}}(\rho,r)\frac{\der}{\der\lambda}j_{\ell}\left(\frac{\lambda r}{h}\right)
\end{equation}
When integrating by parts once we can then combine \eqref{dervbig}, and the Cauchy Schwarz inequality like the proof of \eqref{diag} where we use the bounds in Lemma \ref{matrixbound} and the prefactor from the terms in \eqref{der} for the second term times $B_{\ell}$ to obtain
\begin{align*}
\left|\frac{t}{h} K_{1,-}(y,y',t,\rho,h)\right|\leq CR\left|\frac{t}{h} K^{\Delta_{free}}(y,y,t,\rho,h)\right|\leq \frac{CR}{h^4}
\end{align*}
whenever $Y_{n}=\Psi_{1,\ell,m}$, in the kernel $K_{-}$. Similarly from Lemma \ref{matrixbound}, replacing $\frac{\der}{\der\lambda}B_{1,\ell,\lambda}$ with $\frac{\der}{\der\lambda}B_{2,\ell,\lambda}$ or $\frac{\der}{\der\lambda}B_{3,\ell,\lambda}$, the proof of the bound for the other polarization directions follow analogously. 
\item Case $|\lambda r|>h$, without loss of generality assume $r\geq r'$ and $\ell>\frac{\lambda\rho e}{2h}$. Let the subscript $+$ denote the restriction to this case. We have the large term estimates \eqref{onebig} and \eqref{twobig} as before. 
Then we obtain a bound for the case of $Y_{n}=\Psi_{1,\ell,m}$ only using the bound on $B_{1,\ell,\lambda_0}$ in \eqref{Bbound}. That is the sum over $\ell>\frac{\lambda\rho e}{2h}$ of \eqref{diag} multiplied by $\frac{t}{h}$ can be bounded by
\begin{align*}
&\left|\frac{t}{h} K_{1,+}(y,y',t,\rho,h)\right|
\\
\le& \frac{C\chi_{R}(r)\chi_R(r')}{h^4}\int\limits_0^{\infty}\sum_{\ell>\frac{\lambda\rho e}{2h}}\varphi(\lambda)\lambda^{2}\max\{r,(\ell+1)h\}\left|\Psi_{1,\ell,m}(\theta)j_{\ell}(\lambda_0\rho)(-\rmi)^{\ell}\right|^2\,\der\lambda \\
\le& \frac{CR}{h^4},
\end{align*}
this gives the bound of $K_1$ for \eqref{impcross} when $Y_{n}=\Psi_{1,\ell,m}$. The rest of the cross terms from \eqref{maxkernel} follow similarly using Cauchy-Schwarz. For the second case when $Y_{n}=\Psi_{2,\ell,m}$, we have that in this case the first terms in 
\begin{align*}
&\frac{1}{h}\int\limits_0^{\infty}\sum\limits_{\substack{(\ell,m),\\ \ell>\lambda_0\rho e}}\chi_{R}(r)\chi_R(r')\varphi(\lambda)\left(\frac{\ell(\ell+1)h}{\lambda}+1\right)^2\left|\Psi_{2,\ell,m}(\theta)\frac{(zj_{\ell})'}{\rho}|_{z=\lambda_0\rho}(-\rmi)^{\ell}\right|^2\,\der\lambda 
\\
&\le \frac{C}{h}\int\limits_0^{\infty}\chi_{R}(r)\chi_R(r')\varphi(\lambda)\sum\limits_{\ell\ge 1}\frac{\ell^7}{4^\ell}\,\der\lambda\le \frac{C}{h}\chi_{R}(r)\chi_R(r').
\end{align*}
\end{enumerate}
Combining the three cases and summing over all three components of the kernel, one with $\Psi_{1,\ell,m}$, one with $\Psi_{2,\ell,m}$ and one with $\Psi_{3,\ell,m}$, with the realization the VSH are orthogonal to each other on the sphere, there exists a constant $C$ so that for any $f\in C_0^{\infty}(M_{3R},(M_{3R})^3)$ which is divergence free 
\begin{align}
\|\chi_{R}(r)\chi(r\leq r')K^{\mathcal{O},m}(y,y',\rho,h)f\|_{L^1}\leq C\|f\|_{L^1}\left(h^{-5}\min \left\{1,\frac{R}{t}\right\}\right).
\end{align}
The monotone convergence theorem implies the desired result for the dense class of vector inputs. By density the result holds for $f\in H^1(M_{3R},(M_{3R})^3)$, with $f$ in the domain of the operator as described in the statement of the main theorem. 
\end{proof}

\section{Finite Speed of Propagation}\label{finitespeed} 

We show how to improve our result to a more generic one, Corollary \ref{indepcor}. Assume we have generic data $f\in H^1_{\mathrm{comp}}(M;M^3)$ which is divergence free. We will show that initial data of this form evolves according to three tractable scenarios. 

\begin{itemize}
\item \textbf{Scenario 1:}
Let $\mathcal{O}$ be a ball in 3 dimensions of radius $\rho$. Let $f=f_1\in H^1(M_{3\rho};(M_{3\rho})^3)$. As a result of \cite{melrosescattering}, Theorem 1.6 we only need to prove dispersion for the case $t\leq 6\rho$ (see Scenario 2), which is covered by Theorem \ref{spheretheorem2} with $3R=9\rho$.
\item \textbf{Scenario 2:} We have $f=f_1\in H^1(M_{3\rho};(M_{3\rho})^3)$ and $t> 6\rho$. Here the free space propagator and the obstacle propagator differ only by an error which is exponentially decaying in time \cite{LI}, Lemma 2.6. We review this result here since it needs a small modification to work in the Maxwell scenario.  Let $\Delta$ be the Dirichlet Laplacian in the exterior of the ball $\mathcal{O}$ of radius $\rho$ in three dimensions. We let $u(t,x)$ denote the solution to the Cauchy problem for the wave equation in $M=\mathbb{R}^3\setminus\mathcal{O}$:
\begin{align}\label{cauchy}
&\partial_t^2u-\Delta u=0 \\& \nonumber
u(0,x)=U_0(x)\\& \nonumber
\partial_tu(0,x)=U_1(x)
\end{align}
which is well posed in in $L^1([0,T]; H^1(M))\cap H^1([0,T]; L^2(M))$ for $H^1_0(M)$ Cauchy data. We recall the following definition. Let $T_{R_e}$ be the escape time, that is the time at which there is no geodesic of length $T_{R_e}$ lying entirely within the radius $R_e$ of the obstacle. For smooth Cauchy data we have the following celebrated theorem 
\begin{theorem}[subset of \cite{melrosescattering}, Theorem 1.6]
Consider the Dirichlet laplacian, $\Delta$ in $M$, the exterior of $\mathcal{O}$ the ball of radius $\rho$ in 3 dimensions. Then there exists a sequence $\lambda_j\in \mathbb{C}$, $\mathrm{Im}(\lambda_j)<0$ with $\lim\limits_{j\rightarrow\infty}\Im(\lambda_j))=-\infty$ and associated generalized eigenspaces $W_j\subset C^{\infty}(M)$ of dimensions $m_j<\infty$ such that 
\begin{enumerate}
\item $v\in W_j\Rightarrow v|_{\partial\mathcal{O}}=0$
\item $(\Delta+\lambda_j^2)W_j\subset W_j$, 
\item If $(U_0,U_1)$ are supported in $\{x: |x|\leq R\}$ then there exists $v_{j,k}\in W_j$ such that  for all $\epsilon>0$ and $N\in \mathbb{N}$ and a multi-index $\alpha$ and some constant $C=C(R,N,\epsilon,\alpha)$ the solution to the Cauchy problem for the wave equation \eqref{cauchy} satisfies for $t>T_{R_e}$
\begin{align}
\sup\limits_{\{|x|\leq R_e\}}\left |D_{(t,x)}^{\alpha}\left(u(t,x)-\sum\limits_{j=1}^{N-1}e^{-i\lambda_jt}\sum\limits_{k=0}^{m_j-1}t^kv_{j,k}(x)\right)\right|\leq Ce^{-(t-T_{R_e})(|\mathrm{Im}\lambda_N|-\epsilon)}.
\end{align} 
\end{enumerate}
\end{theorem} 
The parametrix is for the wave equation in \cite{melrosescattering}. If we take $$\mathrm{curl}(x u(t,x))=(\nabla u)\times r\hat{r}$$ then this solves for the electric field in \eqref{system} with $i\lambda \mapsto \partial_t$ and $E|_{\partial\cal{O}}=\nu\times E=0$. We recall here that $\nabla u(t,\rho)=0$ by assumption that $u(t,x)$ solves \eqref{cauchy}. Taking this solution instead does not change the statement of the result of \cite{melrosescattering} Theorem 1.6 due to the presence of the derivative in $\alpha$. By density the result follows for $H_0(\mathrm{curl}, \mathrm{div}\, 0, M)$ fields which are $0$ outside the set $M_{3\rho}$.
\item \textbf{Scenario 3}:
The result of \cite{LP} Chapter 8 is that if our input data $f_2\in C_c^{\infty}(B_{3\rho}^c;(B_{3\rho}^c)^3)$ then the operators $\mathrm{cos}(t\Delta_{\free})-\mathrm{cos}(t\Delta_{\rel})$ coincide on these forms when $|x-\rho|>t$. This implies that they coincide for all $t<2\rho$ acting on this particular class of initial data. Their result is a manifestation of Huygen's principle. By density the result follows for $H(B_{3\rho}^c; (B_{3\rho}^c)^3)$ data which is divergence free. Notice then that when $\varphi(hD_t)$ is present $\varphi(hD_t)\mathrm{cos}(t\Delta_{rel})$ is also bounded as an operator on $H^1(B_{3\rho}^c,(B_{3\rho}^c)^3)$. 

\end{itemize}
We can then consider $f\in C_c^{\infty}(M;M^3)$ as the sum of data $f=f_1+f_2$. The missing region is then when the part of the initial data  $f_2\in H^1(M_{3\rho}^c;(M_{3\rho}^c)^3)$ and in time has evolved beyond $t>2\rho$.  We claim that this region is covered by using data from scenario 1, 2 and 3. Select $R=3\rho$, then since the obstacle is non-trapping the support of data enters the region described by scenario 1) when if $t_0$ our new starting time for evolution of the data is considered as $2\rho$ and those contributions of the support of this data which move into scenario 3) at times $t\geq 2\rho$ already disperse at the correct rate. Thus our dispersive estimate in Theorem \ref{spheretheorem2} can be extended to Corollary \ref{indepcor}. 
This is a generalization of a similar in spirit argument to \cite{LI} which they use to make their dispersive estimates independent of the source and observer and restrict their analysis to a very large ball. Better arguments about the asymptotics of the Bessel and Hankel functions can be used to improve the value of the constant which depends directly on $R$ in Theorem \ref{spheretheorem2} and therefore the overall constant's geometric dependence on $\rho$ but it is not clear that this type of analysis will result in better powers of $h$ for the Maxwell operator. 

While is easy to construct compactly supported data in $H_0(\mathrm{curl},M_R)$, this may be the source of the extra powers of $h$, since this is the not the ideal space for the data according to the Hodge decomposition theorem. Balls with charge, multiple spheres like in \cite{PZ}, and other obstacles  and conductors represent future goals for geometric scattering theory.


\section{Appendix}
\subsection{Hankel Function Estimates} \label{AA}
We recall that in $3d$ we have 
\begin{align}\label{explicithankel1}
h_n^{(1)}(r)=(-\rmi)^{n+1}\frac{e^{\rmi r}}{r}\sum\limits_{m=0}^n\frac{\rmi^m}{m!(2r)^m}\frac{(n+m)!}{(n-m)!}
\end{align}
is the explicit form for the spherical Hankel function of degree $n$. Let $a_0(n)=1$ and define 
\begin{align}
a_{k}(n)=\frac{(4n^2-1)(4n^2-3^2). . . (4n^2-(2k-1)^2)}{k!8^k}=\frac{(\frac{1}{2}-n)_k(\frac{1}{2}+n)_k}{(-2)^kk!}.
\end{align}
From the Nist digital library \cite{olver2010nist} eqs. 10.17.13,14, and 15, we have that 
\begin{align}\label{eqn:Hankelerror}
{H^{(1)}_{n}}\left(z\right)=\left(\frac{2}{\pi z}\right)^{\frac{1}{2}}e^{
\rmi\omega}\left(\sum_{k=0}^{\ell-1}(\rmi)^{k}\frac{a_{k}(n)}{z^{k}}+R_{%
\ell}^{\pm}(n,z)\right)
\end{align}
where $\ell \in \N$, $\omega=z-\tfrac{1}{2}n\pi-\tfrac{1}{4}\pi$ and
\[\left|R_{\ell}^{\pm}(n,z)\right|\leq 2|a_{\ell}(n)|\mathcal{V}_{z,\pm \rmi
\infty}\left(t^{-\ell}\right)\*\exp\left(|n^{2}-\tfrac{1}{4}|\mathcal{V}_{z,
\pm \rmi\infty}\left(t^{-1}\right)\right),\]
where $\mathcal{V}_{z,\rmi\infty}\left(t^{-\ell}\right)$ may be estimated in various sectors as follows
\[\mathcal{V}_{z,\rmi\infty}\left(t^{-\ell}\right)\leq\begin{cases}|z|^{-\ell},&0%
\leq\operatorname{ph}z\leq\pi,\\
\chi(\ell)|z|^{-\ell},&\parbox[t]{224.037pt}{$-\tfrac{1}{2}\pi\leq%
\operatorname{ph}z\leq 0$ or
$\pi\leq\operatorname{ph}z\leq\tfrac{3}{2}\pi$,}\\
2\chi(\ell)|\Im z|^{-\ell},&\parbox[t]{224.037pt}{$-\pi<\operatorname{ph}z\leq%
-\tfrac{1}{2}\pi$ or
$\tfrac{3}{2}\pi\leq\operatorname{ph}z<2\pi$.}\end{cases}\]
Here, $\chi(\ell)$ is defined by
\begin{align*}
 \chi(x) := \pi^{1/2} \Gamma\left(\tfrac{1}{2}x+1\right)/\Gamma \left(\tfrac{1}{2}x+\tfrac{1}{2}\right).
\end{align*}
We get the large order bounds from the definitions above and Stirling's formula. It follows that 
\begin{align}\label{lob}
|J_{n}(z)|\leq \frac{1}{\sqrt{2\pi n}}\left(\frac{ez}{2n}\right)^{n} \qquad 
|Y_{n}(z)|\leq \sqrt{\frac{2}{\pi n}}\left(\frac{ez}{2n}\right)^{-n}. 
\end{align} 

\subsection{Spherical Bessel functions}\label{Abf}
The spherical Bessel functions $j_\ell$ are usually defined as $j_\ell(x)=\sqrt{\frac{\pi}{2x}} J_{\ell + \frac{1}{2}}(x)$. These functions appear when separating variables for the three dimensional Helmholtz equation. We have from \cite{olver2010nist} that for real $z$
\begin{align}\label{lob1}
y_{n}^2(z)+j_{n}^2(z)=|h_{3,n}^{(1)}(z)|^2=\sum\limits_{k=0}^n\frac{(2k)!(n+k)!}{2^{2k}(k!)^2(n-k)!z^{2k+2}}=\sum\limits_{k=0}^n\frac{s_{k}(n+\frac{1}{2})}{z^{2k+2}}.
\end{align}
It follows directly from the above equality \eqref{lob1} that $$|h^{(1)}_{k}(z)|\leq|h^{(1)}_{n}(z)|$$ for all real $z$ and $k\leq n$. The equation which both the spherical Bessel and Hankel functions $w=j_{\ell}(z)$ or $w=h^{1,2}_{\ell}(z)$ satisfy is  
\begin{align}\label{eqdef}
z^2\frac{\der^2w}{\der z^2}+2z\frac{\der w}{\der z}+(z^2-\ell(\ell+1))w=0
\end{align} 
where $\ell$ is a non-negative integer. Furthermore if $f_n$ is either $h_n^{(1)}, h_n^{(2)}$ or $j_n$ then it satisfies the recurrence relation for the derivative with respect to input
\begin{align}\label{recurrence}
\frac{d}{dz}f_{n}(z)=f_{n-1}(z)-\left(\frac{n+1}{z}\right)f_{n}(z).
\end{align}

\subsection{Asymptotic expansions}\label{rellich}

Assume that $f\in C^{\infty}(M; \mathbb{R}^3)$, then if $\lambda\in \mathbb{R}$ is not zero, and $g=(\Delta_{\rel}-\lambda^2)f$ is compactly supported then $f$ is called outgoing for $\lambda$ whenever $f=R_{\lambda}\tilde f$ with $R_\lambda$ being the outgoing resolvent and $\tilde f$ is compactly supported. If $\chi$ is such that $f\chi\in C_0^{\infty}(M; \mathbb{R}^3)$ then $f$ is outgoing if and only if $(1-\chi f)$ is outgoing. This definition depends only on the behavior of $f$ at infinity. The notion is independent of the precise structure of the resolvent and is independent of any ball $B_{R_1}(0)=B_1$, with $R_1>1$ large enough so that $\mathcal{O}\subset \overline{B_{R_1}(0)}$. Therefore $f$ outgoing on $M$ is equivalent to $f|_{\mathbb{R}^3\setminus B_1}$ being outgoing on $\mathbb{R}^3$. This fact can be used to see that an outgoing $f$ has an asymptotic expansion
\begin{align}
f=\frac{e^{\rmi\lambda r}}{r}\Phi+\mathcal{O}\left(\frac{1}{r}\right) \quad r\rightarrow \infty
\end{align}
with $\Phi\in C^{\infty}(\mathbb{S}^{2}; \mathbb{C}^3)$ the restriction of an entire function on $\mathbb{C}^3$ to the sphere. The expansion can be differentiated in $r$. 

In this part we give a short proof of the formula for $g\in L^2(\mathbb{S}^2,\mathbb{C}^3)$
\begin{align}
\frac{1}{2\pi}\int\limits_{(\mathbb{S}^{2})^3}e^{-\rmi\lambda x\cdot\omega}g(\omega)\,\der\omega=2\sum\limits_{n}a_{n}\phi_{n}\left(\frac{x}{r}\right)j_{\ell_{n}}(\lambda r)(-\rmi)^{\ell_{n}}
\end{align}
which also holds with $\phi_{n}$ being replaced by $\Psi_{1,\ell,m}$ in dimension 3. We also have in dimension 3
\begin{align}
\frac{1}{(2\pi)}\int\limits_{(\mathbb{S}^{2})^3}e^{-\rmi\lambda x\cdot\omega}g(\omega)\,\der\omega=2\sum\limits_{\ell,m}a_{2,\ell,m}\Psi_{2,\ell,m}\left(\frac{x}{r}\right)\left.\frac{(yj_{\ell}(y))'}{y}\right|_{y=\lambda r} (-\rmi)^{\ell-1}.
\end{align}
We can expand using the identity
\begin{align}\label{Aa}
&\frac{\der}{\der x}(x j_{\ell}^{(1)}(x))=xj_{\ell-1}(x)-\ell j_{\ell}(x).
\end{align}
We obtain
\begin{align}\label{Ab}
\left.\frac{(zj_{\ell}(z))'}{z}\right|_{z=\lambda r}=j_{\ell-1}(\lambda r)-\frac{\ell}{\lambda r} j_{\ell}(\lambda r).
\end{align} 
Rellich's uniqueness theorem (Theorem 4.18 in \cite{DZbook}) gives that solutions for real $\lambda>0$ are determined uniquely by their asymptotic expansion in any dimension
\begin{align}
f(r\theta)=\frac{e^{\rmi\lambda r}}{r^{\frac{d-1}{2}}}g_{+}(\theta)+\frac{e^{-\rmi\lambda r}}{r^{\frac{d-1}{2}}}g_{-}(\theta)+\mathcal{O}\left(\frac{1}{r^{\frac{d+1}{2}}}\right).
\end{align}
Application of the stationary phase lemma gives generically in any dimension: 
\begin{equation*}
\int\limits_{(\mathbb{S}^{d-1})^3}e^{-\rmi\lambda x\cdot\omega}g(\omega)\,\der\omega= \left(\frac{2\pi}{\lambda r}\right)^{\frac{d-1}{2}}\left(e^{-\rmi \lambda r}e^{\frac{\rmi(d-1)\pi}{4}}g(\theta)+e^{\rmi\lambda r}e^{\frac{-\rmi(d-1)\pi}{4}}g(-\theta)\right) +\mathcal{O}\left((\lambda r)^{-\frac{d+1}{2}}\right)
\end{equation*}
as $\lambda r\rightarrow \infty$. Here the remainder terms depend on the norm $$\|g(\omega)\|_{H^2(\mathbb{S}^{2})^3}$$ by stationary phase whenever $|\lambda r|$ is large. These expansion terms can be differentiated term by term in $x$ and then expanded again using stationary phase to obtain the large $|\lambda r|$ terms. Comparing the spherical Bessel function terms allows to match the coefficients in the generalized Eigenfunction expansions, using equation \eqref{eqn:Hankelerror}. Notice that the terms associated to $\Psi_{3,\ell,m}$ are genuinely lower order in $r$ as $\lambda r\rightarrow \infty$.

We recall Proposition C.4 in \cite{OS}
\begin{prop}[subset of Prop. C.4 in \cite{OS}]\label{OSprop}
Let $\chi\in C^{\infty}(M)$ be supported in $\mathbb{R}^3\setminus B_1$ be such that $1-\chi\in C_0^{\infty}(M)$. Then $f\in C_0^{\infty}(M)$ is outgoing if and only if $\chi f$ is outgoing for the Laplace operator on $\mathbb{R}^3$. 
\end{prop}
We use this to show the following 
\begin{lem}\label{replacebase}
In the exterior of the ball of radius $\rho$, $f$ is outgoing for fixed $\lambda$ and divergence free then $f$ is of the form 
\begin{align}
\tilde{h}^{(1)}_{\lambda,Y}(Y).
\end{align}
\end{lem}
\begin{proof}
By repeating the proof of Theorem \ref{K2.43} (Theorem 2.43 in \cite{KH}) using $r\hat{r}h_{\ell}^{(1)}(\lambda r)$ instead of $r\hat{r}j_{\ell}^{1}(\lambda r)$ we have a free space solution of the Maxwell system \ref{system}, as remarked in \cite{KH}. Now we need to compare this to the exact solution. The free space kernel for the scalar Helmholtz equation is in three dimensions 
\begin{align}
R_{0,\lambda}(x,y)=\frac{e^{\rmi\lambda|x-y|}}{4 \pi |x-y|}.
\end{align} 
The leading order asymptotic terms coincide with the free space kernel of using the expansions in the previous part of this appendix subsection with $j_{\ell}(\lambda r)$ replaced by $h_{\ell}^{(1)}(\lambda r)$ and using equation \eqref{eqn:Hankelerror} and \eqref{recurrence}. Indeed we see that for every $n\in \mathbb{N}$ and $z\in \mathbb{C}$ that 
\begin{align}
h_n^{(1)}(z)=\frac{e^{i(z-\frac{\pi}{2}(n+1))}}{z}\left(1+O(\frac{1}{|z|})\right)  \quad |z|\rightarrow \infty
\end{align}
and also 
\begin{align}
\frac{d}{dz}h_n^{(1)}(z)=\frac{e^{i(z-\frac{\pi}{2}n)}}{z}\left(1+O(\frac{1}{|z|})\right) \quad |z|\rightarrow \infty
\end{align}
uniformly with respect to $\frac{z}{|z|}$, which concludes the proof. (The equivalence is standard, see Definition 4.16, interpretation eq 4.4.4, with Theorem 3.37 in \cite{DZbook}) . 
\end{proof}

\subsection{Vector Spherical Coordinates}\label{VSH}
This short section is a review of some necessary facts about spherical coordinate systems. It is taken directly from example A.17 in \cite{KH}, page 328. We consider the sphere of radius $r$. We can parametrize the boundary of a ball of radius $r$ in local coordinates as follows
\begin{align}
\sigma(\theta,\phi)=r(\sin\theta\cos\phi,\sin\theta\sin\phi,\cos\theta)^t.
\end{align}
The surface gradient and the surface divergence respectively on the boundary of the sphere are given by 
\begin{align*}
&\mathrm{Grad} f(\theta,\phi)=\frac{1}{r}\frac{\partial f}{\partial\theta}(\theta,\phi)\hat{\theta}+\frac{1}{r\sin\theta}\frac{\partial f}{\partial\varphi}(\theta,\phi)\hat{\varphi} \\&
\mathrm{Div} F(\theta,\phi)=\frac{1}{r\sin\theta}\frac{\partial}{\partial\theta}(\sin\theta F_{\theta}(\theta,\phi))+\frac{1}{r\sin\theta}\frac{\partial F_{\phi}}{\partial\phi}(\theta,\phi)
\end{align*}
where $\hat{\theta}=(\cos\theta\cos\phi,\cos\theta\sin\phi,-\sin\theta)^t$ and $\hat{\phi}=(-\sin\phi,\cos\phi,0)^t$ are the vectors which span the tangent plane and $F_{\theta},F_{\phi}$ are the components of $F$ with respect to these vectors, that is $F=F_{\theta}\hat{\theta}+F_{\phi}\hat{\phi}$. When spherical coordinates are used $x=r\hat{r}\in\mathbb{R}^3$, the surface differential operators with respect to the unit sphere are used instead of the operators on the boundary, $\partial D$ of $$D=\{x\in\mathbb{R}^3: |x|=r\}.$$ We indicate this by using the index $\mathbb{S}^2$ for the differential operators with respect to the sphere with $r=1$. By scaling invariance we have 
\begin{align}
\mathrm{Grad}_{\mathbb{S}^2}f(r,\hat{r})=r\mathrm{Grad}f(r\hat{r}) \qquad \mathrm{Div}_{\mathbb{S}^2}F(r,\hat{r})=r\mathrm{Div}F(r\hat{r}).
\end{align}
It is understood that $f(r,\cdot)$ is a function of the unit sphere on the right hand side and on the left hand side it is a function on $\partial D$ on the right hand side. This scaling allows us to create divergence free generalized eigenfunctions for Maxwell's equations using a basis of vector spherical harmonics. \\

We also recall here that we have the associated Legendre functions $P_{\nu}^{\mu}(\pm x)$ or $P_{\nu}^{-\mu}(\pm x)$ solves the equation given by (14.21.1) in \cite{olver2010nist}
\begin{align}\label{legendre}
\left(1-x^{2}\right)\frac{{\mathrm{d}}^{2}w}{{\mathrm{d}x}^{2}}-2x\frac{%
\mathrm{d}w}{\mathrm{d}x}+\left(\nu(\nu+1)-\frac{\mu^{2}}{1-x^{2}}\right)w=0.
\end{align}
The functions $P_n^m(\cos\theta)e^{\pm \rmi m\varphi}$ with $0\leq m\leq n$ are harmonic functions on $|r|=1$ and are homogeneous polynomials of degree $n$. (c.f. Theorem 2.10 in \cite{KH}, for a proof). \\ 
\\

\textbf{Acknowledgements:} The authors wish to thank Alexander Strohmaier for discussions during the course of this paper. Part of this research was carried while the author A. W. was at the Erwin Schr\"odinger Institute Thematic Program "Spectral Theory and Mathematical Relativity." A.W. wishes to thank the Erwin Schr\"odinger Institute for its hospitality during this time. The author Y-L F. was supported by EPSRC postdoctoral fellowship EP/V051636/1. The author A.W. wishes to thank Oana Ivanovici for editorial comments as well as the anonmyous referees.

\end{document}